%% file: AdaptiveConvolutions.tex
\documentclass[11pt]{article}
\input{Packages}

\input{MacroFile1}

\begin{document}
\title{Adaptive Convolutions}
\author{Ilja Klebanov\footnote{Zuse Institute Berlin (ZIB), Takustra\ss e 7, 14195 Berlin, Germany (klebanov@zib.de).}}
\date{\today}
\maketitle
\begin{abstract}
When smoothing a function $f$ via convolution with some kernel, it is often
desirable to adapt the amount of smoothing locally to the variation of $f$. For
this purpose, the constant smoothing coefficient of regular convolutions needs
to be replaced by an adaptation function $\mu$. This function is
matrix-valued which allows for different degrees of smoothing in different directions.
The aim of this paper is twofold. The first is to provide a theoretical
framework for such adaptive convolutions. The second purpose is to derive a
formula for the automatic choice of the adaptation function $\mu = \mu_f$ in
dependence of the function $f$ to be smoothed.
This requires the notion of the \emph{local variation} of $f$, the
quantification of which relies on certain phase space transformations of $f$.
The derivation is guided by meaningful axioms which, among other things, guarantee invariance of adaptive convolutions under shifting and scaling of $f$.
%
%
%
%
%
%

\end{abstract}




\input{sections/Motivation}
\input{sections/Theory}
\input{sections/ContinuityEquation}

\input{sections/ChoiceAdaptationFunction}

\input{sections/Examples}

\input{sections/Conclusion}

\appendix

\input{sections/OtherTypes}
\input{sections/Proofs}

\begin{acknowledgement}
\ \\
I thank Caroline Lasser for many insightful and motivating discussions.
%
%
%
\end{acknowledgement}


----------------------------------------------------------------------------------------
\bibliographystyle{abbrv}

\bibliography{myBibliography}

%
%
%
%
%

\end{document}

%% file: Packages.tex
\usepackage{amsfonts,amssymb,amsmath,amsthm,amstext,amssymb,mathtools}
\usepackage{caption,subcaption,float,color,graphicx,enumerate}
\usepackage{dsfont}  
\usepackage{fullpage} 
\usepackage{parskip} 
\usepackage{paralist} 

\usepackage{array}  
\newcolumntype{C}[1]{>{\centering\arraybackslash}m{#1}}

\usepackage{tikz}
\usetikzlibrary{bayesnet,er,trees,shapes.symbols,mindmap,arrows,snakes,shapes.misc,shapes.arrows,chains,matrix,positioning,scopes,decorations.pathmorphing,patterns}
\usepackage{tikz-cd} 

\pgfdeclarelayer{bg}    
\pgfsetlayers{bg,main}  

%% file: MacroFile1.tex
\usepackage{mathabx} 
\mathchardef\ordinarycolon\mathcode`\:
\mathcode`\:=\string"8000
\begingroup \catcode`\:=\active
\gdef:{\mathrel{\mathop\ordinarycolon}}
\endgroup

\newcommand*{\Cdot}[1][2.4]{%
  \mathpalette{\CdotAux{#1}}\cdot%
}
\newdimen\CdotAxis
\newcommand*{\CdotAux}[3]{%
  {%
    \settoheight\CdotAxis{$#2\vcenter{}$}%
    \sbox0{%
      \raisebox\CdotAxis{%
        \scalebox{#1}{%
          \raisebox{-\CdotAxis}{%
            $\mathsurround=0pt #2#3$%
          }%
        }%
      }%
    }%
    \dp0=0pt %
    \sbox2{$#2\bullet$}%
    \ifdim\ht2<\ht0 %
      \ht0=\ht2 %
    \fi
    \sbox2{$\mathsurround=0pt #2#3$}%
    \hbox to \wd2{\hss\usebox{0}\hss}%
  }%
}

\DeclarePairedDelimiter\abs{\lvert}{\rvert}%
\DeclarePairedDelimiter\norm{\lVert}{\rVert}%
\makeatletter
\let\oldabs\abs
\def\abs{\@ifstar{\oldabs}{\oldabs*}}
\let\oldnorm\norm
\def\norm{\@ifstar{\oldnorm}{\oldnorm*}}
\makeatother

\definecolor{dunkelgruen}{rgb}{0,0.4,0}

\DeclareMathOperator{\diver}{div}

\DeclareMathOperator{\supp}{supp}

\def\R{\mathbb{R}}
\def\N{\mathbb{N}}

\def\C{\mathbb{C}}

\def\E{\mathbb{E}}

\def\Cov{\mathbb{C}\mathrm{ov}}
\def\P{\mathbb{P}}
\def\Id{{\rm Id}}

\theoremstyle{plain}
\newtheorem{theorem}{Theorem}
\newtheorem{definition}[theorem]{Definition}
\newtheorem{proposition}[theorem]{Proposition}

\newtheorem{corollary}[theorem]{Corollary}

\newtheorem{notation}[theorem]{Notation}
\newtheorem{example}[theorem]{Example}

\newtheorem{axiom}[theorem]{Axiom}
\newtheorem{remark}[theorem]{Remark}

\newtheorem*{acknowledgement}{Acknowledgements}

\def\mtm{\mathtt{m}}

\def\F{\mathcal{F}}

\def\gldr{\mathrm{GL}(d,\R)}
\def\tr{{\rm tr}}
\DeclareMathOperator{\stern}{\ast}

%% file: sections/Motivation.tex
\section{Motivation}
\label{section:convolutionMotivation}
The convolution of two integrable functions $f,g\colon\R^d\to\R$,
\[
(f\ast g)(x) = \int_{\R^d}f(y)\, g(x-y)\, \mathrm dy,
\]
is a basic mathematical tool with applications in probability theory, image
processing, optics, acoustics and many other areas.
If $g$ is a probability density, say a Gaussian density
\begin{equation}
\label{equ:gaussian}
g(x) = (2\pi\sigma^2)^{-d/2}\, \exp\left(-\frac{\|x\|^2}{2\sigma^2}\right),
\end{equation}
the convolution $(f\ast g)(x)$ evaluated at a point $x$ can be viewed as the
average over all values $f(y),\ y\in\R^d,$ weighted by $g(x-y)$, i.e. the contribution of $f(y)$ to $(f\ast g)(x)$ decreases as the distance between $x$ and $y$ increases. Convolutions are
therefore often used to `flatten' or `smooth' a
function $f$ and $g$ is called a \emph{smoothing kernel} in this case.

A natural question arising here is how to choose the standard deviation $\sigma$ of
$g$, i.e. how strongly we want to smooth the function $f$. Roughly speaking, the
aim is usually to flatten out the bumps and edges without losing the shape of
the function entirely.

\begin{example}
\label{example:twoFunctions}
Assume that we decided that some $\sigma>0$ is adequate to smooth a function
$f_1\colon\R^d\to\R$, and $f_2(x) = f_1(\alpha x)$ is  a scaled version of $f_1$ by some factor $\alpha>0$.
In order to achieve a similar smoothing effect for $f_2$, 
the density $g$ has to be scaled accordingly, $\tilde g =
\alpha^d g(\alpha x)$ (the prefactor $\alpha^d$ is just a normalization
factor), as visualized in Figure \ref{fig:chooseKernel}:
\[
(f_2 \ast \tilde g)(x)
=
\alpha^d \int_{\R^d} f_1(\alpha y)\, g(\alpha (x-y))\, \mathrm dy 
=
\int_{\R^d} f_1(y)\, g(\alpha x-y)\, \mathrm dy 
=
(f_1\ast g)(\alpha x).
\]
\begin{figure}[H]
        \centering
        \begin{subfigure}[b]{0.48\textwidth}
                \centering
                \includegraphics[width=1\textwidth]{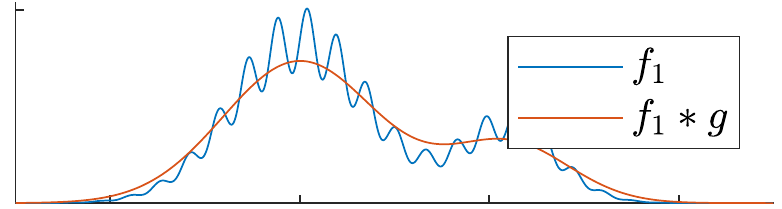}
                \caption{The standard deviation of the Gaussian $g$ from
                \eqref{equ:gaussian} chosen appropriately to smooth the function
                $f_1$, here $\sigma = 0.4$. }
                \end{subfigure}%
        \hfill
        \begin{subfigure}[b]{0.48\textwidth}
                \centering
                \includegraphics[width=1\textwidth]{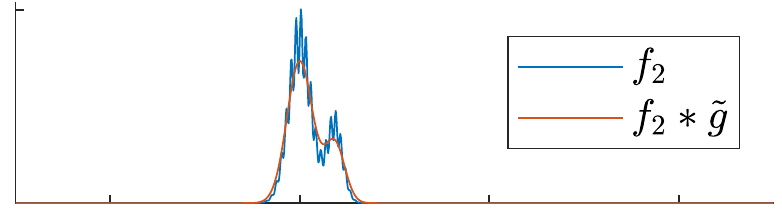}
                \caption{In order to smooth $f_2(x) = f_1(\alpha x)$, the
                function $g$ has to be scaled by the same factor to get an
                analogous result (here $\alpha = 6$).}
        \end{subfigure}        
        \label{convolution1und2}
        \caption{Choosing proper standard deviations of the density $g$ to
        smooth differently scaled versions of the function $f_1$.}
        \label{fig:chooseKernel}
\end{figure}
\end{example}
In other words, once the `degree' or `extent' of smoothing is agreed upon, the
width of the smoothing kernel $g$ has to be adapted to the `variation' of the function $f$.
However, if the variation of the function changes considerably in space, no
single suitable width $\sigma$ can be found and one is forced to adapt it \emph{locally},
\begin{equation}
\label{equ:firstDef}
(f\ast_{\mu} g) (x):= \int f(y)\, \abs{\det\mu(y)}\, g\big(\mu(y)(x-y)\big)\,
\mathrm dy\, ,
\end{equation}
where $\mu:\R^d\to \mathrm{GL} (d,\R)$ is a measurable (matrix-valued) function, which scales the density $g$ locally by different factors $\mu(y)$.
\begin{example}
\label{example:differingVariation}
Taking up our functions $f_1$ and $f_2$ from Example \ref{example:twoFunctions}, we build up a new function
\[
f(x) = f_1(x) + f_2(x-a),
\]
separating $f_1$ and $f_2$ in space by a shift $a\in\R^d$.
Choosing $g$ as a smoothing kernel would be inappropriate for one `part' of
the function (oversmoothing), as choosing $\tilde g$ would be for the other
(undersmoothing), see Figure \ref{fig:convolution3und4und5} (a) and (b).
For the shift $a = 8$ and scaling factor $\alpha = 6$, the adaptive convolution $f\ast_{\mu} g$ with
\begin{equation}
\label{equ:choiceMu}
\mu(y) = \begin{cases}
1 & \text{if } x<7,\\
6 & \text{if } x\ge 7,
\end{cases}
\end{equation}
provides a suitable solution for both parts (see Figure \ref{fig:smoothingh3}).
\begin{figure}[H]
        \centering
        \begin{subfigure}[b]{\textwidth}
                \centering
                \includegraphics[width=1\textwidth]{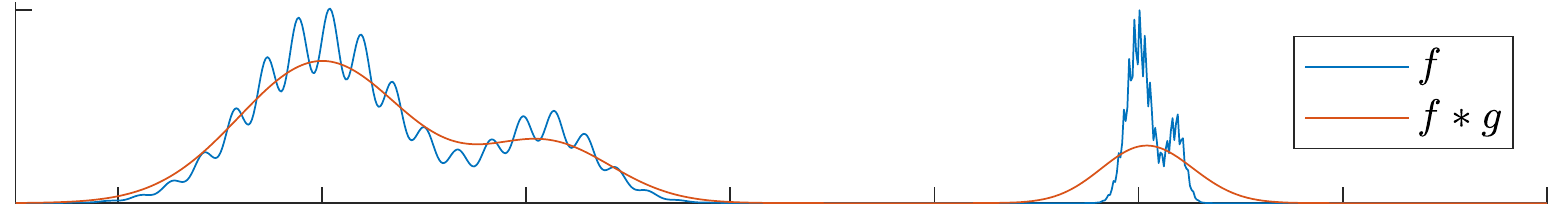}
                \caption{$g$ is an appropriate smoothing kernel for the left
                part, but not for the right one.}
                \label{fig:smoothingh1}
                \end{subfigure}%
        \vspace{0.2cm}
        \begin{subfigure}[b]{\textwidth}
                \centering
                \includegraphics[width=1\textwidth]{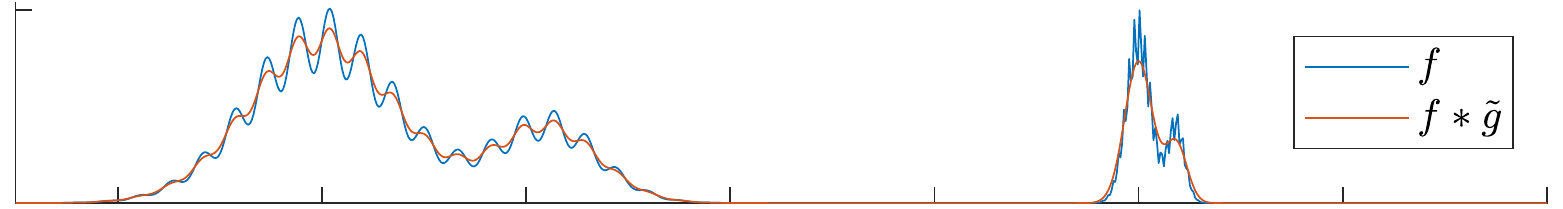}
                \caption{$\tilde g$ is an appropriate smoothing kernel for the
                right part, but not for the left one.}
                \label{fig:smoothingh2}
        \end{subfigure}        
        \vspace{0.2cm}
        \begin{subfigure}[b]{\textwidth}
                \centering
                \includegraphics[width=1\textwidth]{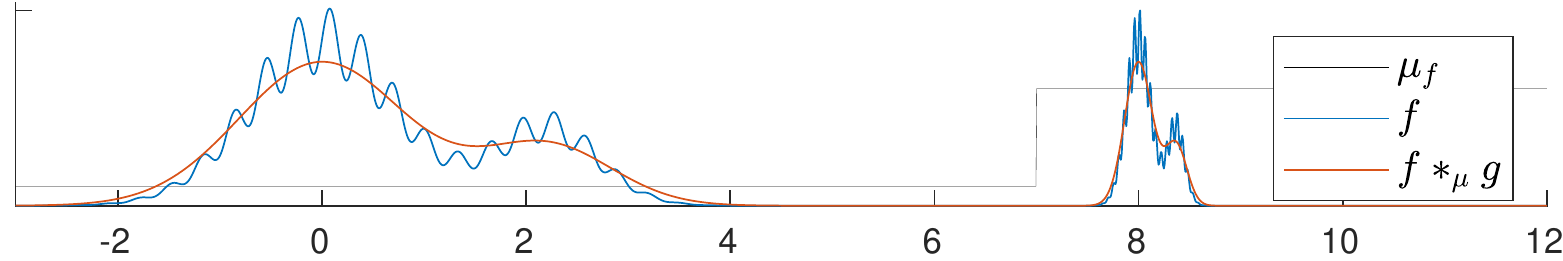}
                \caption{Adaptive convolutions guarantee an appropriate width
                of the smoothing kernel everywhere. $\mu$ is chosen as in equation \eqref{equ:choiceMu}.}
        \label{fig:smoothingh3}
        \end{subfigure}
        \caption{Comparison of common convolutions with adaptive convolutions. Here, the shift is $a=8$ and the scaling factor is $\alpha = 6$.}
        \label{fig:convolution3und4und5}
\end{figure}
\end{example}
While in this example the adaptation function $\mu$ was chosen \emph{manually},
the question arises on how this choice can be automatized -- what is
a good adaptation function $\mu = \mu_f$ in dependence of $f$? We wish to
address this issue by first imposing proper axioms concerning the behavior
of $\mu_f$ under shifting and scaling of $f$, which guarantee invariance of the
adaptive convolution under such transformations. In order to derive a formula
for $\mu_f$ that fulfills these axioms, we will introduce a measure for the
local variation of $f$ which relies on certain phase space transformations of
$f$.

One possible application of the adaptive convolution framework is variable
kernel density estimation, which is discussed in a companion paper \cite{klebanovVKDE}.


This paper is structured as follows.
In Section \ref{section:Theory}, the framework and theory of adaptive convolutions is presented in a slightly more general setup. Two versions of Young's inequality as well as a differentiation rule and a continuity equation for adaptive convolutions are discussed.
In Section \ref{section:ConvolutionAdaptationFunction}, we address the second
main issue of this paper -- the \emph{automatic} choice of the adaptation function $\mu$.
In Appendix \ref{section:Convolution2}, further approaches to adaptive smoothing
are discussed. The proofs are provided in Appendix \ref{section:proofs}.

%% file: sections/Theory.tex
\section{Theoretical Properties of Adaptive Convolutions}
\label{section:Theory}

\begin{definition}[adaptive convolutions, adaptation function]
\label{def:adaptiveConvolution}
We define the \emph{generalized convolution} $f\bar\ast G$ of two
measurable functions $f\colon \R^d\to\R$ and $G \colon\R^d\times\R^d\to\R$ as
the integral operator with kernel $G$:
\[
(f\bar\ast G)(x) := \int f(y)\, G(x,y)\, dy,
\]
whenever the integral is well-defined. Let $f\in L^1(\R^d),\
g\in L^p(\R^d),\ 1\le p\le\infty$, $\mu\colon \R^d\to \mathrm{GL}(d,\R)$ be a measurable function and
\[
g_{\mu,p}(x,y):= \abs{\det \mu(y)}^{1/p}\,
g\big(\mu(y)(x-y)\big),
\]
where $1/p := 0$ for $p=\infty$. We define the \emph{$\mu$-adaptive convolution}
of $f$ with $g$ by
\[
f\ast_{\mu}^p g := f\bar\ast g_{\mu,p}\, .
\]
$\mu$ will be called \emph{adaptation function}.
In the case $p=1$, the definition coincides with \eqref{equ:firstDef} and we
will omit the index $p$: $f\ast_\mu g:= f\ast_\mu^1 g$ and $g_\mu :=
g_{\mu,1}$.
\end{definition}

\begin{remark}
\label{rem:obviousObservations}
\begin{compactenum}[(a)]
  \item We will allow $\mu(x)\notin\gldr$ and even attain infinite values in the
  nodes of $f$, i.e. for $x\notin\supp (f)$, since this does not influence the
  integral (using the convention $0\cdot\infty := 0$).
  \item The adaptive convolution is \emph{not} symmetric and the notation
  $f\ast_{\mu}g$ indicates that $g$ is scaled by $\mu(y)$, while $f\sideset{_\mu}{} \stern g$ can be used, if $f$ is to be scaled (we will not need the second notation).
  \item The adaptive convolution is \emph{not} associative.
  \item The adaptive convolution is linear in both arguments.
  \item The $\mu$-adaptive convolution reduces to the common convolution $f\ast
  g$ for $\mu \equiv \Id$.
  \item In Proposition \ref{prop:ConvDer}, we will slightly abuse the notation of adaptive convolutions
by applying it to matrix-valued functions $f\in L^1(\R^d,\R^{m\times n}),\
g\in L^p(\R^d,\R^{n\times \ell})$ or to the case where $f$ is an operator acting on $g$.
The definitions go analogously.  
\end{compactenum}
\end{remark}

\subsection{Young's Inequality for Adaptive Convolutions}
In the following, we will discuss a weak and a strong version of Young's
inequality \cite{young1912multiplication}, \cite[Theorem
3.9.4]{bogachev2007measure} and the conditions under which they hold. This will result in the generalization of
Young's inequality for convolutions to the case of adaptive convolutions.

\begin{theorem}
\label{theorem:young1}
Let $f\in L^1(\R^d)$, $1\le p\le \infty$ and $G \colon\R^d\times\R^d\to\R$ be
measurable such that $\left\| G(\Cdot,y) \right\|_p\le \Gamma$ for some $\Gamma \ge 0$ independent of
$y\in\R^d$. Then
\[
\|f\bar\ast G\|_{p}\le
\|f\|_{1}\, \Gamma\, .
\]
\end{theorem}

This suffices to prove the weak version of Young's inequality for $\mu$-adaptive convolutions:
\begin{corollary}[Young's
inequality]
\label{cor:young}
Let $f\in L^1(\R^d),\ g\in L^p(\R^d),\ 1\le p\le\infty$ and $\mu\colon
\R^d\to\mathrm{GL}(d,\R)$ be a measurable function. Then $f\ast_{\mu}^p g\in L^p(\R^d)$ and
\[
\|f\ast_{\mu}^p g\|_{p}\le
\|f\|_{1}\, \|g\|_{p}\, .
\]
\end{corollary}

Under the additional assumption that $\left\|G(x,\Cdot)\right\|_p\le\Gamma$ is
also bounded by $\Gamma$ for each $x\in\R^d$ (e.g. if $G$ is symmetric), a stronger version of Young's inequality holds:

\begin{theorem}
\label{theorem:young2}
Let $1\le p,q,r\le\infty$ such that $1+\frac{1}{r} = \frac{1}{p} + \frac{1}{q}$.
Let $f\in L^q(\R^d)$ and $G \colon\R^d\times\R^d\to\R$ be measurable such that
$\left\| G(\Cdot,y) \right\|_p\le \Gamma$ for all $y\in\R^d$ and $\left\|
G(x,\Cdot) \right\|_p\le \Gamma$ for all $x\in\R^d$ for some $\Gamma \ge 0$.
Then
\[
\|f\bar\ast G\|_{r}\le
\|f\|_{q}\, \Gamma\, .
\]
\end{theorem}

\begin{remark}
In the particular case $p=1$ we also have
\begin{align*}
\int_{\R^d} \left(f\ast_{\mu}g\right)(x) \, \mathrm dx
&=\int_{\R^d} f(y) \int_{\R^d} g_\mu(x,y)\, \mathrm dx \, \mathrm dy
=
\left(\int_{\R^d} f(y) \, \mathrm dy \right)\left(\int_{\R^d} g(x)\,
\mathrm dx\right)
\intertext{and, if $g\colon \R^d\to\R$ is a probability density function,}
\|f\ast_\mu g\|_1
&=
\int_{\R^d} |f(y)| \int_{\R^d} g_\mu(x,y)\, \mathrm dx \, \mathrm dy
=
\|f\|_1\, .
\end{align*}
\end{remark}

\subsection{Derivatives of Adaptive Convolutions}
Similar to common convolutions, there are slightly modified (but
non-symmetric!) rules for the differentiation of adaptive convolutions.
This will require some notation:
\begin{notation}
\label{notation:derivativeMultilinear}
We will
use the standard multi-index notation for $\alpha =
(\alpha_1,\dots,\alpha_d)\in\N_0^d$ and $x\in\R^d$,
\begin{align*}
|\alpha|
&:=
\alpha_1 + \cdots + \alpha_d\, ,
\\
\partial^{\alpha}
&:=
\partial_{x_1}^{\alpha_1}\partial_{x_2}^{\alpha_2}\cdots
\partial_{x_d}^{\alpha_d}\, ,
\\
x^{\alpha}
&:=
x_1^{\alpha_1}x_2^{\alpha_2}\cdots x_d^{\alpha_d}\, .
\end{align*}
If $g\in\C^m(\R^d)$, $m\in\N$, then its $k$-th derivative ($k\le m$) can be
viewed as a symmetric multilinear map $D^{k}g\colon \left(\R^d\right)^k\to\R$.
Further, $k$ vectors $v_1,\dots,v_k\in\R^d$ give rise to a linear form on the space
$\mathcal{SML}^k\left(\R^d\right)$ of
k-fold symmetric multilinear forms on $\R^d$,
\[
[v_1,\dots,v_k]\colon \mathcal{SML}^k\left(\R^d\right) \to\R,\qquad
\phi\mapsto \phi\left(v_1,\dots,v_k\right),
\]
which is just the identification of the Banach space
$\left(\R^d\right)^k$ with its double dual.
Finally, for $\alpha\in\N^d$ and a matrix $M\in\R^{d\times d}$ with columns
$M_{\bullet,1},\dots,M_{\bullet,d}$, $\alpha(M)$ will denote the following
$|\alpha|$-tuple of vectors in $\R^d$:
\[
\alpha(M) :=\big[
\underbrace{M_{\bullet,1},\dots,M_{\bullet,1}}_{\alpha_1\ \mathrm{times}},
\underbrace{M_{\bullet,2},\dots,M_{\bullet,2}}_{\alpha_2\ \mathrm{times}},\dots,
\underbrace{M_{\bullet,d},\dots,M_{\bullet,d}}_{\alpha_d\ \mathrm{times}}\big]
\in \mathcal L \left(\mathcal{SML}^{|\alpha|}\left(\R^d\right),\R\right).
\]
Since it will be viewed as a linear map acting on \emph{symmetric} multilinear
forms as described above, the order of the vectors does not matter.
\end{notation}

\begin{proposition}
\label{prop:ConvDer}
Let $f\in L^1(\R^d),\ g\in C^m(\R^d)$ for some $m\in\N$,
$\mu\colon \R^d\to\gldr$ be a bounded measurable function and $1\le
p\le\infty$ such that $\partial^{\alpha} g\in L^p(\R^d)$ for all $\alpha\le m$.
\\
Then $f\ast_{\mu}^p g\in C^m(\R^d)$ and for all $\alpha\in\N^d$ with
$|\alpha|\le m$, the derivative $\partial^{\alpha}\left(f\ast_{\mu}^p
g\right)\in L^p(\R^d)$ is given by  (we slightly abuse the notation as mentioned in Remark \ref{rem:obviousObservations} (f))
\[
\partial^{\alpha}\left(f\ast_{\mu}^p g\right)
=
\left(f\cdot \alpha(\mu)\right)\ast_{\mu}^p\, D^{|\alpha|}g\, .
\]
\end{proposition}

%% file: sections/ContinuityEquation.tex
\subsection{Continuity Equation for Convolutions and Adaptive Convolutions}
\label{section:ConvolutionCont}
Assume that $(\rho_t)_{t\ge 0}$ is a time-dependent probability density, which
fulfills the continuity equation
\[
\partial_t \rho_t = -\diver(\rho_t v_t) = -\diver(j_t)
\]
for some velocity field $v_t\colon\R^d\to\R^d$ (or current $j_t\colon\R^d\to\R^d$).
Assume further that we want to smooth $\rho_t$ by considering an
adaptive convolution $\rho_{g,t} = \rho_t\ast_{\mu_t} g$ with a time-dependent
adaptation function $\left(\mu_t\right)_{t\ge 0}$ (since we are dealing with
probability density functions, $p=1$ here).
How does the continuity equation have
to be modified in order to describe the evolution of $\rho_{g,t}$?
We were surprised to find an explicit formula for the modified continuity
equation:

\begin{proposition}
\label{prop:convolutionContinuity2}
Let $\rho_t\in L^1(\R^d)$ be a time-dependent probability
density function, which fulfills the continuity equation
\[
\partial_t\rho_t + \diver j_t = 0
\hspace{1.5cm} (t\in \R)
\]
for some current $j_t\in L^1(\R^d,\R^d)$, such that
$(\rho_t,\, j_t)_{t\in\R}\in C^1\left(\R^{1+d},\R^{1+d}\right)$.
\\
Further, let $g\in L^p\cap C^1(\R^d)$ for some $1\le p <\infty$,
\[
\gamma (x) :=
x\, g(x)
,\qquad
N_t(x) = 
\begin{pmatrix}
j_t(x)^{\intercal} \left(D_x
\left(\mu_{t}\right)_{1,\Cdot}^{\intercal}\right)^{\intercal}\hspace{-0.1cm}(x)\,
\\ \vdots \\
j_t(x)^{\intercal} \left(D_x
\left(\mu_{t}\right)_{d,\Cdot}^{\intercal}\right)^{\intercal}\hspace{-0.1cm}(x)\,
\end{pmatrix}
\in\R^{d\times d}
\hspace{1.5cm}
(x\in\R^d),
\]
such that $g,\, \gamma$ and all their first derivatives are
bounded: $\displaystyle g,\,
\gamma\in W^{1,\infty}$.
\\
Finally, let $(\mu_t)_{t\in\R}\in
C^2\left(\R^{1+d},\gldr\right)$, such that for
each $i,j=1,\dots,d,\ t\in\R$ 
\[
\left(\mu_t\right)_{i,j},\,
\nabla\left(\mu_t\right)_{i,j},\, \partial_t\left(\mu_t\right)_{i,j},\, 
\left(\mu_t^{-1}\right)_{i,j}\,  \in L^{\infty}.
\]
Then
\[
\rho_{g,t} : = \rho_t\ast_{\mu_t} g
\hspace{1.5cm} (t\in \R)
\]
is a probability density function, which fulfills the continuity equation
\[
\boxed{
\partial_t\rho_{g,t} =  - \diver j_{g,t}
\qquad\text{for}\qquad j_{g,t}
=
j_t\ast_{\mu_t} g - \left[\mu_t^{-1}\left(N_t + \rho_t\, \partial_t\mu_t\right)
\mu_t^{-1}\right] \ast_{\mu_t}\gamma\, .
}
\]
Further, $(\rho_{g,t},\, j_{g,t})_{t\in[0,\infty)} \in
C^1\left(\R^{1+d},\R^{1+d}\right)$.
\end{proposition}

\begin{corollary}
Under the assumptions of Proposition \ref{prop:convolutionContinuity2}, if
$\rho_{g,t} = \rho_t\ast g_{A_t}$
is the common convolution, but with time-dependent scaling matrix
$\left(A_t\right)_{t\in\R} \in C^2\left(\R,\gldr\right)$ of the smoothing
kernel,
\[
g_{A_t}(x) = |\det A_t|\, g(A_t x),
\hspace{2cm}
\gamma_{A_t}(x) = g_{A_t}(x)\, A_t\, x,
\]
then the probability density function $\rho_{g,t}$ solves the continuity
equation
\[
\partial_t\rho_{g,t} =  - \diver j_{g,t}
\qquad\text{for}\qquad j_{g,t}
=
j_t\ast g_{A_t} - \left[\rho_t\, A_t^{-1}\, \partial_t A_t\, A_t^{-1}\right]
\ast\gamma_{A_t}\, .
\]
Further, $(\rho_{g,t},\, j_{g,t})_{t\in[0,\infty)} \in
C^1\left(\R^{1+d},\R^{1+d}\right)$.
\end{corollary}

%% file: sections/ChoiceAdaptationFunction.tex
\section{Automatic Choice of the Adaptation Function $\mu$}
\label{section:ConvolutionAdaptationFunction}

The adaptation function $\mu$ in Example \ref{example:differingVariation} was
chosen \emph{manually} for the adaptive smoothing of a function $f\in
L^1(\R^d)$.
Let us now discuss how this choice can be performed \emph{automatically} in
dependence of the function $f$ that we want to smooth. To this end, we will have
to get a grip on the local variation of $f$, an issue that we will address by means of certain phase space transforms
introduced in the following subsection.

Finding a good dependence for the adaptation function $\mu = \mu_f$ on the
function $f$ is a difficult task and will only be partially answered here.
We will justify the implicit formula
\[
\mu_f^2(x) = \frac{\left( \nabla
f\nabla f^{\intercal} - f\, D^2 f\right) \ast
G_{(\lambda\mu_f)^{-2}(x)}^2}{(2-\lambda^2)\, f^2 \ast
G_{(\lambda\mu_f)^{-2}(x)}^2}(x)\, ,
\]
where $0<\lambda<\sqrt{2}$, but neither prove uniqueness, nor any kind of
optimality.
Also, we will assume that $f\not\equiv 0$ lies in the Sobolev space
$W^{2,2}(\R^d,\R)$, and restrict ourselves to radially symmetric smoothing kernels $g$, i.e.
$g(x) = \gamma(\|x\|_2^2)$ for some function $\gamma\colon \R_{\ge 0}\to\R$.
\begin{remark}
The square root $M^{1/2} = \sqrt{M}$ of a symmetric and positive definite matrix
$M\in\R^{d\times d}$ will denote the unique symmetric and positive definite matrix $N\in\R^{d\times
d}$ such that $N^2 = M$ (see \cite[Theorem 7.2.6]{horn2012matrix}).
\end{remark}
Let us first gather some conditions, which we would like our adaptation function
$\mu_f$ to fulfill (see also the motivation in Section
\ref{section:convolutionMotivation}):
\begin{axiom}[Adaptation Axioms]
\label{cond:adaptation}
Let $\mathcal M = \{\mu\colon \R^d\to\gldr\colon \mu \text{ measurable}\}$.
We say that a mapping
$$
\mtm\colon W^{2,2}(\R^d,\R)\to \mathcal M,
\qquad
f\mapsto \mu_f,
$$
fulfills the \emph{Adaptation Axioms}, if for any $a\in\R^d$,
$\alpha\in\R\setminus\{0\}$, $A\in\gldr$, any parametrized function
$f^{(t)} = \sum_{k=1}^K f_k(\Cdot-a_k^{(t)})$, $t\ge 0,$ with $f_k\in
W^{2,2}(\R^d,\R)$, $a_k^{(t)}\in\R^d$, such that
$\|a_k^{(t)}-a_j^{(t)}\|\xrightarrow{t\to\infty}\infty$ for all $k\neq j$, and
any $x\in\R^d$,
\begin{enumerate}
  \item[(A1)]
  $\displaystyle
  \mu_{f(\Cdot\, - a)}(x) = \mu_f(x-a)$ \hfill (invariance under shifting),
  \item[(A2)]
  $\displaystyle \mu_{\alpha\cdot f} = \mu_f$ \hfill (invariance under scalar multiplication),
  \item[(A3)]
  $\displaystyle \mu_{f(A\cdot\, \Cdot)}^\intercal(x)\, \mu_{f(A\cdot\, \Cdot)}(x)
  =
  A^{\intercal}\, \mu_f^\intercal(A x)\,\mu_f(A x)\, A$
  \hfill (invariance under scaling),
  \item[(A4)]
  $\mu_{f^{(t)}}(x+a_k^{(t)}) \xrightarrow{t\to\infty}\mu_{f_k}(x)$ for all
  $k=1,\dots,K$
  \hfill
  (locality),
  \item[(A5)]
  $\mu_f(x)$ should describe some kind of variation of
  $f$ locally around $x\in\R^d$.
\end{enumerate}
\end{axiom}
While axiom (A5) is rather subjective, axioms (A1)--(A3) are chosen such
that the $\mu_f$-adaptive convolution behaves nicely under shifting and scaling
of $f$. Axiom (A4) guarantees that a sum $f = \sum_{k=1}^K f_k$
of several functions $f_1,\dots,f_K$ with `far apart' supports is smoothed in
approximately the same way as these functions would have been smoothed
separately, $f\ast_{\mu_f} g \approx \sum_{k=1}^K f_k\ast_{\mu_{f_k}}g$, see
also the motivating Example \ref{example:differingVariation}.
These consequences are stated in the following theorem:
\begin{theorem}
\label{theorem:AdaptationConditions}
Assuming Adaptation Axioms \ref{cond:adaptation} (A1)--(A4) and adopting that
notation, we have for any $f\in W^{2,2}(\R^d,\R)$, radially symmetric $g\in
L^p$ and $x\in\R^d$:
\begin{enumerate}[(i)]
  \item $\displaystyle (f(\Cdot\, - a)\ast^p_{\mu_{f(\Cdot\, - a)}} g) (x)=
  (f\ast^p_{\mu_f}g) (x-a)$
  \hfill
  \small (shifted function $\Rightarrow$ shifted convolution)\normalsize,
  \item $\displaystyle (\alpha f)\ast^p_{\mu_{\alpha f}} g =
  \alpha (f\ast^p_{\mu_f}g) $
  \hfill
  \small (stretched function $\Rightarrow$ stretched convolution)\normalsize,
  \item $\displaystyle (f(A\cdot\Cdot)\ast^p_{\mu_{f(A\cdot\Cdot)}} g) (x)=
  (f\ast^p_{\mu_f}g) (Ax)$
  \hfill
  \small (scaled function $\Rightarrow$ scaled convolution)\normalsize.
  \item $(f^{(t)}\ast_{\mu_{f^{(t)}}} g)(x) = \sum_{k=1}^K
  (f_k\ast_{\mu_{f_k}}^p g)(x-a_k^{(t)})$ asymptotically for $t\to\infty$,
  \\
  more precisely:  
  $(f^{(t)}\ast_{\mu_{f^{(t)}}}^p g)(x+a_k^{(t)})
  \xrightarrow{t\to\infty}(f_k\ast_{\mu_{f_k}}g)(x)$
  \hfill
  \small (locality)\normalsize.
\end{enumerate}
\end{theorem}

\begin{remark}
Adaptation Axiom \ref{cond:adaptation} (A3) is a necessary detour around the
more appealing condition
\[
\mu_{f(A\cdot\, \Cdot)}(x) = \mu_f(Ax)A,
\]
since the product of a symmetric and positive definite matrix with an invertible
matrix is in general no longer symmetric and positive definite. This is also
the reason why we have to restrict ourselves to radially symmetric smoothing kernels $g$. Roughly speaking, covariance
matrices are easier to treat than their square roots.
\end{remark}

One possible choice that fulfills the Adaptation Axioms
\ref{cond:adaptation} (A1)--(A4) is
\[
\mu_f^{(a)} = \sqrt{\frac{\nabla f\, \nabla f^{\intercal}}{f^2}}.
\]
However, $\nabla f\nabla f^{\intercal}$ is only positive \emph{semi}-definite
and therefore $\mu_f^{(a)}\in\gldr$ could be violated. Also, it
is unclear in how far the Adaptation Axiom (A5) is fulfilled.
Obviously, this last axiom is not rigorous and we will discuss it now. In order
to get a grasp on it, we will make use of three transformations
explained in the following subsection.

\subsection{Phase Space Transformations}
\label{section:PhaseSpace}

In this section, we will discuss four transformations, which will allow us to quantify the (local) variation of a function.
All four transformations can be viewed from various perspectives and we will
focus on the time-frequency point of view. In the following, $\mathcal
S(\R^d,\C)$ will denote the Schwartz space of rapidly decreasing functions.

\begin{definition}[Gaussian, Fourier transform, windowed Fourier transform,
adaptive windowed Fourier transform, Wigner transform]
We define
\begin{compactenum}[(i)]
  \item the \emph{Gaussian function} $G[a,\Sigma]\colon \R^d\to\R$ with mean
  $a\in\R^d$ and symmetric and positive definite covariance matrix
  $\Sigma\in\gldr$ as well as the abbreviation $G_\Sigma
  := G[0,\Sigma]$,
  \item the \emph{Fourier transform} $\F: \mathcal S(\R^d,\C)\to \mathcal
  S(\R^d,\C)$,
  \item the \emph{windowed Fourier transform} $\F_{Q}:
  \mathcal S(\R^d,\C)\to \mathcal S(\R^d\times \R^d,\C)$ with window
  width
  \footnote{Typically, windowed Fourier transforms are defined for isotropic
  covariance matrices, i.e. $Q = \sigma\Id$. Using arbitrary covariance
  matrices allows for different window sizes in different directions.}
  $Q\in\gldr$, $Q^{\intercal} = Q$,
  \item the \emph{adaptive windowed Fourier transform} $\F_{Q}:
  \mathcal S(\R^d,\C)\to \mathcal S(\R^d\times \R^d,\C)$ with variable
  window width $Q\colon \R^d\to\gldr$, $Q^{\intercal} = Q$,
  \item the \emph{Wigner transform} $W: \mathcal S(\R^d,\C)\to \mathcal
  S(\R^d\times\R^d,\R)$ by
\footnote{Usually, the Wigner transform
is defined by
$
W(f,g)(x,\xi) = (2\pi)^{-d}\int_{\R^d}\overline{f\left(x+\frac{y}{2}\right)}\,
g\left(x-\frac{y}{2}\right)\, e^{iy^\intercal\xi}\, \mathrm dy.
$
We will only use its definition on the diagonal, $Wf := W(f,f)$, where it is
real-valued, which can be seen by applying the transformation $y\mapsto -y$ to
the defining integral.}
\end{compactenum}
\begin{align}
\label{equ:generalGaussian}
G[a,\Sigma](x)
&=
(2\pi)^{-d/2}\, |\det \Sigma|^{-1/2}\,
\exp\left[-\frac{1}{2}(x-a)^{-\intercal}\Sigma^{-1} (x-a)\right],
\\
\mathcal F f (\xi)
&=
(2\pi)^{-d/2} \int_{\R^d} f(y)\, e^{-iy^\intercal\xi}\,
\mathrm dy\, ,
\\
\label{equ:windowedFourier}
\F_{Q} f(x,\xi)
&=
\pi^{-d/4} \int_{\R^d}f(y)\, G_{Q^2}(x-y)\,
e^{-iy^\intercal\xi}\, \mathrm dy\, ,
\\
\label{equ:adaptiveWindowedFourier}
\F_{Q} f(x,\xi)
&=
\pi^{-d/4} \int_{\R^d}f(y)\, G_{Q^2(x)}(x-y)\,
e^{-iy^\intercal\xi}\, \mathrm dy\, ,
\\
Wf(x,\xi)
&=
(2\pi)^{-d}\int_{\R^d}\overline{f\left(x+\frac{y}{2}\right)}\,
f\left(x-\frac{y}{2}\right)\, e^{iy^\intercal\xi}\, \mathrm dy.
\end{align}
\end{definition}

We will restrict ourselves to presenting only a few properties of these
transforms, further properties can be found in e.g. \cite{folland1989harmonic}.

\begin{proposition}[Plancherel Theorem and Fourier Inversion Formula]
\label{prop:fourierInverse}
\ \\
The Fourier transform is an isometric isomorphism on $\mathcal S(\R^d,\C)$ with
inverse given by
\[
\mathcal F^{-1} f(x) = (2\pi)^{-d/2} \int_{\R^d} f(\xi)\,
e^{ix^\intercal\xi}\, \mathrm d\xi\, .
\]
\end{proposition}
\begin{proof}
See \cite{folland1989harmonic}.
\end{proof}
For $f\in\mathcal S(\R^d,\C)$, the Fourier inversion formula implies
\[
(2\pi)^{-d} \int_{\R^d} f(x) \int_{\R^d} e^{-ix^\intercal\xi}\, \mathrm d\xi\,
\mathrm dx \ =\ 
(2\pi)^{-d/2} \int_{\R^d} \F f(\xi)\,  e^{i0^\intercal\xi}\, \mathrm d\xi\, 
\ =\ 
\F^{-1}\F f(0)
\ =\ 
f(0). 
\]
The technique of using $(2\pi)^{-d}\int_{\R^d} e^{-ix^\intercal\xi}\, \mathrm
d\xi$ as a $\delta$-distribution will be used in the proofs of Propositions
\ref{prop:muFourier}, \ref{prop:muFBI}, \ref{prop:AdaptiveMuFBI} and
\ref{prop:muWigner2}.

From the point of view of time-frequency analysis, the Fourier transform yields a
decomposition of the signal $f$ into its frequencies: for each frequency $\xi$,
it indicates the extent of its occurrence in $f$.
However, one is often interested in the local frequencies of $f$, meaning, which
frequencies of $f$ occur at (or close to) a specific
point in time $x$.
This can be analyzed using a windowed Fourier transform, also called Gabor transform, which does not
`see' the values and frequencies of $f$ far from $x$. Applied to each
point in time $x$, this defines the mapping \eqref{equ:windowedFourier} on the
phase space.

\begin{figure} 
\centering
\begin{subfigure}[b]{0.13\textwidth}
	\centering	
	\includegraphics[width=3.235\textwidth,angle = 	90]{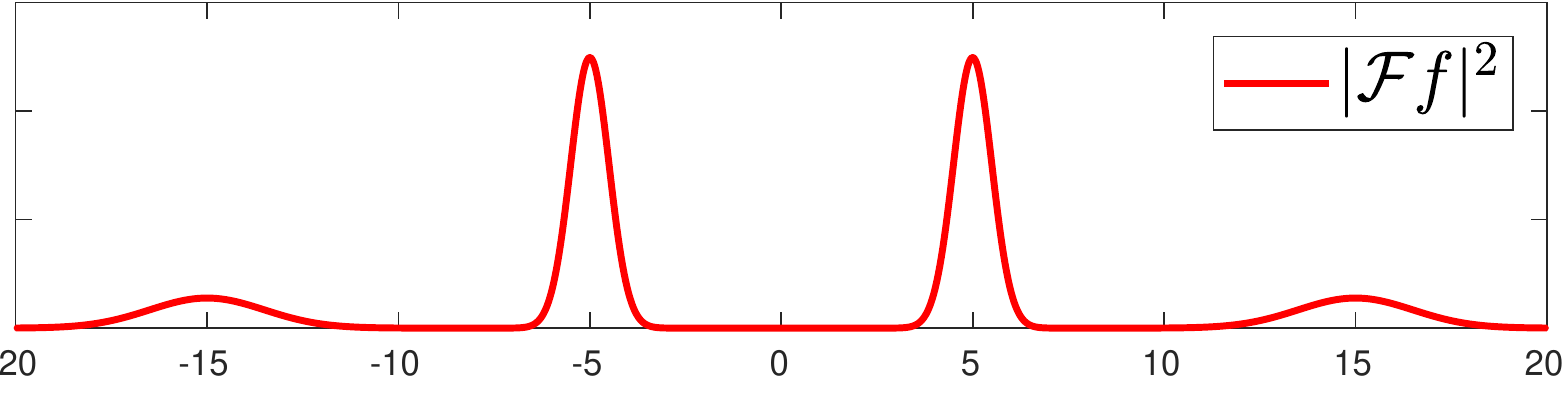}
\end{subfigure}
\hfill
\begin{subfigure}[b]{0.42\textwidth}
	\centering	                
	\includegraphics[width=\textwidth]{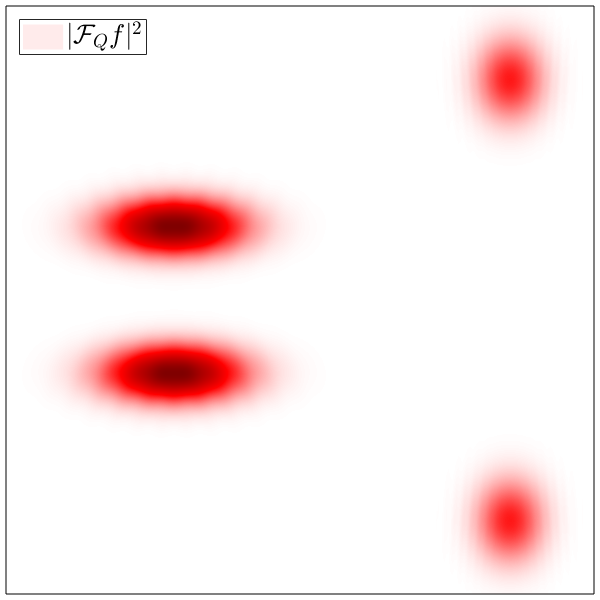}
\end{subfigure}
\hfill
\begin{subfigure}[b]{0.42\textwidth}
	\centering	                
	\includegraphics[width=\textwidth]{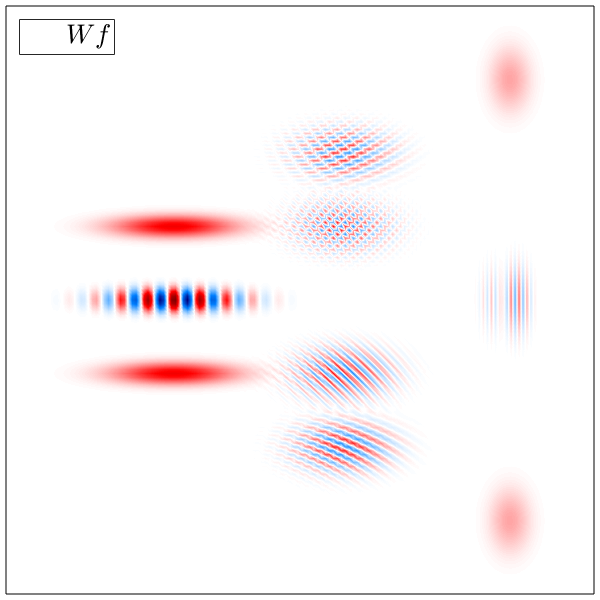}
\end{subfigure}
\vfill
\vspace{0.3cm}
\hspace{0.13\textwidth}
\hfill
\begin{subfigure}[b]{0.42\textwidth}
	\centering
	\includegraphics[width=\textwidth]{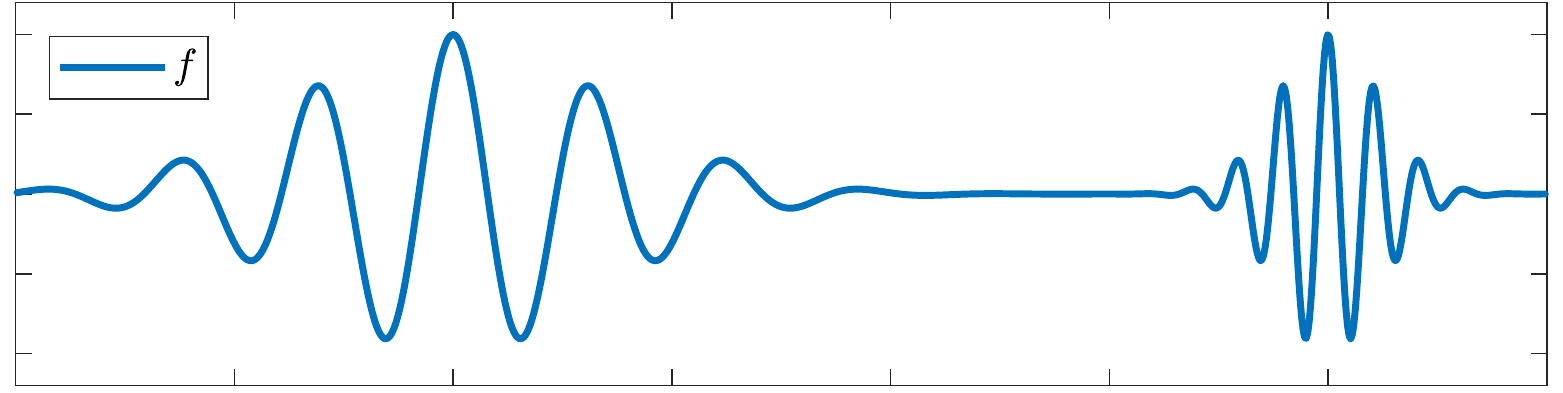}
\end{subfigure}
\hfill
\begin{subfigure}[b]{0.42\textwidth}
	\centering
	\includegraphics[width=\textwidth]{images/Function_for_Fourier}
\end{subfigure}
\caption{Visualization of how different frequencies of a function $f$ are
represented by the Fourier transform $\F f$, the windowed Fourier transform
$\F_{Q}f$ and the Wigner transform $Wf$ (only the moduli squared of $\F f$
and $\F_{Q}f$ are plotted).}
\end{figure}

The width $Q$ of the Gaussian function $G_{Q^2}$ is a
double-edged sword: The smaller it is chosen, the more accurate the
time-frequency description becomes in the $x$-direction, since only values that
are very close to the considered time $x$ contribute to $\F_{Q}
f(x,\Cdot)$.
However, the smaller the window, the more  `difficult' it becomes to determine
the frequencies in this small time period and the more `blurred' the frequency
decomposition $\F_{Q} f(x,\Cdot)$ becomes in $\xi$-direction.
This issue is a manifestation of the so-called uncertainty principle, see the
discussion in \cite[Chapter 2]{grochenig2001foundations}.

The adaptive windowed Fourier transform \eqref{equ:adaptiveWindowedFourier}
allows to perform this trade-off differently in different regions of $\R^d$ by
choosing the width to be $x$-dependent.

At the expense of losing positivity, the Wigner transform $Wf$ of $f$ provides
a way to `deblurr' $|\F_{Q} f|^2$, resulting in the `correct' marginal
densities, as stated by the following proposition:
\begin{proposition}
\label{prop:husimi}
Let $f\in\mathcal S(\R^d,\C)$ and $\displaystyle G_{\sigma}(x,\xi) =
\left(2\pi\sigma^2\right)^{-d/2} \exp\left(-\frac{\|x\|^2+\|\xi\|^2}{2\sigma^2}\right)$
denote a Gaussian in phase space. Then the Wigner transform $Wf$ of $f$ fulfills
\begin{align*}
&\int_{\R^d} Wf(x,\xi)\, \mathrm d\xi\
\ =\  
|f(x)|^2\, ,
\qquad
\int_{\R^d} Wf(x,\xi)\, \mathrm dx
\ =\
|\F f(\xi)|^2\, ,
\intertext{and}
&Wf \ast G_{\sqrt{1/2}} = |\mathcal \F_{1} f|^2.
\end{align*}
\end{proposition}

\begin{proof}
See \cite[(3.10)]{hillery1997distribution}.
\end{proof}

\subsection{Transformations and the Adaptation Function $\mu$}
We are now ready to construct an adaptation function $\mu_f$ that fulfills the
Adaptation Axioms \ref{cond:adaptation}.
Let us start with a \emph{global} version.
A natural way to describe the global variation of a function $f\in
L^1(\R^d,\R)$ is to consider its spectral density $\rho\propto|\F f|^2$, since
functions with high oscillations tend to have high values of $|\F f|^2$ away
from the origin.

The spectral density also has the proper behaviour under scaling of $f$ -- if
$f$ is scaled by some factor $\alpha\neq 0$, $\tilde f(x) = f(\alpha x)$, the
(global) variation is scaled by $\alpha^{-1}$ and in fact the Fourier transform
(and thereby the spectral density) is scaled accordingly:
\begin{equation}
\label{equ:scaleFourier}
\F\tilde f (\xi)
=
(2\pi)^{-d/2} \int_{\R^d} f(\alpha y)\, e^{-iy^\intercal\xi}\,
\mathrm dy
=
\frac{(2\pi)^{-d/2}}{\alpha^d} \int_{\R^d} f(y)\, e^{-iy^\intercal\xi/\alpha}\,
\mathrm dy
=
\alpha^{-d}\F f(\alpha^{-1}\xi)
\, .
\end{equation}
In order to assign a value for the variation to a function $f$, we will
therefore consider the expectation value and covariance defined in the
following proposition:
\begin{proposition}
\label{prop:muFourier}
Let $f\in W^{2,2}(\R^d,\R)\setminus\{0\}$. The expectation value and covariance
matrix of the probability distribution $\P_\rho$ given by the spectral density
\[
\rho = \frac{|\F f|^2}{\|\F f\|_{L^2}^2} = \frac{|\F
f|^2}{\|f\|_{L^2}^2}
\]
(here we used the Plancherel theorem \ref{prop:fourierInverse}) are:
\[
\E_{\rho} = 0
\qquad\text{and}\qquad
\Cov_\rho = \frac{\int_{\R^d} \left(\nabla f\, \nabla
f^{\intercal}\right)(z)\, \mathrm dz}{\|f\|_{L^2}^2}\, .
\]
\end{proposition}

The adaptation function $\mu_f$, which in this global setting is just a constant
adaptation matrix $\mu_f\in\gldr$, can now be assigned the square root of the
covariance matrix,
\[
\boxed{
\mu_f^{(b)} := \sqrt{\Cov_\rho}
=
\frac{\sqrt{\int_{\R^d} \left(\nabla f\, \nabla
f^{\intercal}\right)(z)\, \mathrm dz}}{\|f\|_{L^2}}\, ,
}
\]
and the Adaptation Axioms \ref{cond:adaptation} (A1)--(A3) can easily be
verified (in a global, $x$-independent sense).
However, we are not interested in a \emph{global}, but in a \emph{local}
adaptation. Therefore, we will study the `local frequencies' of $f$ by taking
its windowed Fourier transform $\F_{Q} f(x,\xi)$ instead of its Fourier
transform.

Again, let us consider the expectation value and covariance of the corresponding
probability density in $\xi$:
\begin{proposition}
\label{prop:muFBI}
Let $f\in W^{2,2}(\R^d,\R)\setminus\{0\}$ and $Q\in\gldr$, $Q^{\intercal} = Q$.
Then, for each $x\in\R^d$, the expectation value and covariance matrix of the
probability distribution $\P_{\rho_x}$ given by the density
\[
\rho_x(\xi) = \frac{|\F_Q f(x,\xi)|^2}{\|\F_Q f (x,\Cdot)\|_{L^2}^2}
\]
are:
\[
\E_{\rho_x} = 0
\qquad\text{and}\qquad
\Cov_{\rho_x} = \frac{Q^{-2}}{2} + \frac{\left( \nabla f\nabla
f^{\intercal} - f\, D^2 f\right) \ast G_{Q^2}^2}{2 f^2 \ast
G_{Q^2}^2}(x)\, .
\]
\end{proposition}

Again, we can set
\begin{equation}
\label{equ:FBIChoiceMu}
\boxed{
\mu_f^{(c)}(x) := \sqrt{\Cov_{\rho_x}} = \sqrt{\frac{Q^{-2}}{2} +
\frac{\left( \nabla f\nabla f^{\intercal} - f\, D^2 f\right) \ast
G_{Q^2}^2}{2 f^2 \ast G_{Q^2}}(x)}\ .
}
\end{equation}
However, while the Adaptation Axioms \ref{cond:adaptation} (A1), (A2) and (A4) are
fulfilled, the scale invariance (A3) is violated. The reason for this is
that the window width $Q$ does not scale
with the local variation of $f$ --- a formula analogous to
(\ref{equ:scaleFourier}) does not hold for windowed Fourier transforms.
One might try to adapt $Q$ locally by using the adaptive windowed Fourier
transform \eqref{equ:adaptiveWindowedFourier} but this would require a priori
knowledge of the local variation of $f$, which we are trying to find in the
first place.

We will discuss two different solutions for this problem:
\begin{enumerate}
  \item[(A)] One might avoid the circular reasoning described
  above by a fixed point approach (and a fixed point iteration in practical
  applications):
	\begin{compactitem}
	\item Derive a formula analogous to \eqref{equ:FBIChoiceMu} for adaptive
	windowed Fourier transforms.
	\item  Choose the local width $Q(x)\in\gldr$ proportional to $\mu_f^{-1}(x)$,
	resulting in an implicit formula for $\mu_f$.	
	\end{compactitem}
  \item[(B)] Since the lack of scale invariance is caused by the (constant) width of the
window, or, in other words, by the blurry way we look at the function, we will
`deblurr' it by replacing the term $|\F_Q f(x,\xi)|^2$ in the
probability density $\rho_x$ from Proposition \ref{prop:muFBI} with the Wigner
transform $Wf(x,\xi)$. This replacement is motivated by the discussion in
Section \ref{section:PhaseSpace} and by Proposition \ref{prop:husimi} in
particular.
Since $Wf$ can take negative values, we will consider $|Wf|^2$ instead of $Wf$,
which is a probability density function, if properly normalized.
\end{enumerate}

Let us start with approach (A).

\begin{proposition}
\label{prop:AdaptiveMuFBI}
Let $f\in W^{2,2}(\R^d,\R)\setminus\{0\}$ and $Q\colon \R^d\to\gldr$. Then,
for each $x\in\R^d$, the expectation value and covariance matrix of the probability distribution $\P_{\rho_x}$ given
by the density
\[
\rho_x(\xi) = \frac{|\F_Q f(x,\xi)|^2}{\|\F_Q f (x,\Cdot)\|_{L^2}^2}
\]
are:
\[
\E_{\rho_x} = 0
\qquad\text{and}\qquad
\Cov_{\rho_x} = \frac{(Q^\intercal Q)^{-1}(x)}{2} + \frac{\left( \nabla
f\nabla f^{\intercal} - f\, D^2 f\right) \ast G_{(Q^\intercal Q)(x)}^2}{2 f^2 \ast
G_{(Q^\intercal Q)(x)}^2}(x)\, .
\]
\end{proposition}

Setting $\mu_f(x) := \sqrt{\Cov_{\rho_x}}$ as before and choosing $Q(x) =
(\lambda \mu_f)^{-1}(x)$, $0<\lambda <\sqrt{2}$, as discussed in approach (A)
above, we arrive at an implicit formula for $\mu_f = \mu_f^{(d)}$ (without
loss of generality we assume $\mu_f$ to be symmetric and positive definite),
\begin{equation}
\label{equ:adaptiveFBImu}
\boxed{
\mu_f^2(x) = \frac{\lambda^2 \mu_f^2(x)}{2} + \frac{\left( \nabla
f\nabla f^{\intercal} - f\, D^2 f\right) \ast
G_{(\lambda\mu_f)^{-2}(x)}^2}{2\, f^2 \ast
G_{(\lambda\mu_f)^{-2}(x)}^2}(x)\, .}
\end{equation}
This formula can be simplified by combining the term on the left-hand side with
the first term on the right-hand side, but we prefer this form, because it
guarantees that the right-hand side is positive definite by Proposition
\ref{prop:AdaptiveMuFBI}, which is essential for the fixed point iteration
\eqref{equ:MuFBIFPI}.
\begin{remark}
Here and in the following we will assume that the function $f$ is such that
\eqref{equ:adaptiveFBImu} has a unique symmetric and positive definite solution
$\mu_f$ and that the corresponding fixed point iteration
\begin{equation}
\label{equ:MuFBIFPI}
\mu_f^{(n+1)}(x) = \sqrt{\frac{(\lambda \mu_f^{(n)})^2(x)}{2} + \frac{\left(
\nabla f\nabla f^{\intercal} - f\, D^2 f\right) \ast
G_{(\lambda\mu_f^{(n)})^{-2}(x)}^2}{2\, f^2 \ast
G_{(\lambda\mu_f^{(n)})^{-2}(x)}^2}(x)}
\end{equation}
converges to that solution for every symmetric and positive definite choice
$\mu_f^{(0)}$. The class of functions $f$ for which this is the case remains an
open problem.
\end{remark}
The crucial advantage of the choice \eqref{equ:adaptiveFBImu} is that it
fulfills the Adaptation Axioms \ref{cond:adaptation}:
\begin{theorem}
\label{theorem:AdaptationConditionsFulfilledAdaptiveMu}
Let $\mu_f^{(d)}$ be the solution of the implicit equation
\eqref{equ:adaptiveFBImu}. Then it fulfills the
Adaptation Axioms \ref{cond:adaptation}(A1)--(A4).
\end{theorem}

Now let us discuss approach (B), which suggests to replace the windowed Fourier
transform in Proposition \ref{prop:muFBI} by the modulus squared $\abs{Wf}^2$ of
the Wigner transform $Wf$.
In contrast to the windowed Fourier transform, a formula analogous to
\eqref{equ:scaleFourier} holds for the Wigner transform $Wf$ and thereby for
$|Wf|^2$, which is a promising property for scale invariance (again, $\tilde
f(x) := f(\alpha x)$ for some $\alpha \neq 0$):
\begin{align*}
W\tilde f(x,\xi)
&=
(2\pi)^{-d} \int_{\R^d}\overline{f\left(\alpha x+\frac{\alpha
y}{2}\right)}\, f\left(\alpha x-\frac{\alpha y}{2}\right)\, e^{iy^\intercal\xi}\, \mathrm dy
\\
&=
\frac{(2\pi)^{-d}}{\alpha^d} \int_{\R^d}\overline{f\left(\alpha
x+\frac{y}{2}\right)}\, f\left(\alpha x-\frac{y}{2}\right)\,
e^{iy^\intercal\xi/\alpha}\, \mathrm dy
\\
&=
\alpha^{-d}\, W f(\alpha x,\xi/\alpha),
\\[0.3cm]
\implies\ |W\tilde f|^2(x,\xi)
&=
\alpha^{-2d}\, |Wf|^2(\alpha x,\xi/\alpha).
\end{align*}
As before, let us compute the expectation value and covariance of the corresponding probability density in $\xi$:
\begin{proposition}
\label{prop:muWigner2}
Let $f\in W^{2,2}(\R^d,\R)\setminus\{0\}$. Then, for each $x\in\R^d$, the
expectation value and covariance matrix of the probability distribution $\P_{\rho_x}$ given
by the density
\[
\rho_x(\xi) = \frac{|Wf|^2(x,\xi)}{\|W f (x,\Cdot)\|_{L^2}^2}
\]
are:
\[
\E_{\rho_x} = 0
\qquad\text{and}\qquad
\Cov_{\rho_x} = \frac{\left(f^2\right)\ast\left(\nabla f\, \nabla f^{\intercal}
- f\, D^2 f\right)}{4 \left(f^2\right)\ast\left(f^2\right)}\, (2x).
\]
\end{proposition}

This time, if we choose 
\begin{equation}
\label{equ:finalChoiceMu}
\boxed{
\mu_f^{(e)}(x) := \sqrt{\Cov_{\rho_x}} = \sqrt{\frac{\left(f^2\right)\ast\left(\nabla f\, \nabla f^{\intercal}
- f\, D^2 f\right)}{4 \left(f^2\right)\ast\left(f^2\right)}\, (2x)}\, ,
}
\end{equation}
the Adaptation Axioms \ref{cond:adaptation}(A1)--(A3) are fulfilled, as stated
by the following proposition, but not (A4), as demonstrated by Example
\ref{example:threeGaussians} .
\begin{theorem}
\label{theorem:AdaptationConditionsFulfilled}
$\mu_f^{(e)}$ as defined by \eqref{equ:finalChoiceMu} fulfills the
Adaptation Axioms \ref{cond:adaptation}(A1)--(A3).
\end{theorem}

We will now state one more result, which suggests a possibility to calibrate the original width of the smoothing kernel $g$:
\begin{corollary}
\label{cor:gaugeWidth2}
If $f = G[a,\Sigma]$ is a Gaussian density given by \eqref{equ:generalGaussian},
then we have
\[
\mu_f^{(d)} \equiv \tfrac{1}{\sqrt{2-\lambda^2}}\, \Sigma^{-1/2},
\qquad
\mu_f^{(e)} \equiv \tfrac{1}{2}\, \Sigma^{-1/2}.
\]
\end{corollary}
Therefore, in order to gauge the overall extent of the smoothing process,
Gaussian functions $f$ (or just the standard Gaussian) are well suited for the
calibration of the original width of the smoothing kernel $g$.

%% file: sections/Examples.tex
\subsection{Examples}

\begin{example}
The application of \eqref{equ:FBIChoiceMu}, \eqref{equ:adaptiveFBImu} and \eqref{equ:finalChoiceMu} to the function $f(x) = f_1(x) + f_2(x-a)$ from Example \ref{example:differingVariation} yields the results presented in Figure \ref{fig:convolution6und7und8}. The scale invariance (Adaptation Axiom \ref{cond:adaptation}(A3)) of the choices \eqref{equ:adaptiveFBImu} and \eqref{equ:finalChoiceMu} is clearly visible: $f_2(x) = f_1(\alpha x)$ and, accordingly, $\mu_f$ is $\alpha=6$ times higher in the `right' domain than in the `left' one.

\begin{figure}[H]
        \centering
        \begin{subfigure}[b]{0.8\textwidth}
                \centering
                \includegraphics[width=1\textwidth]{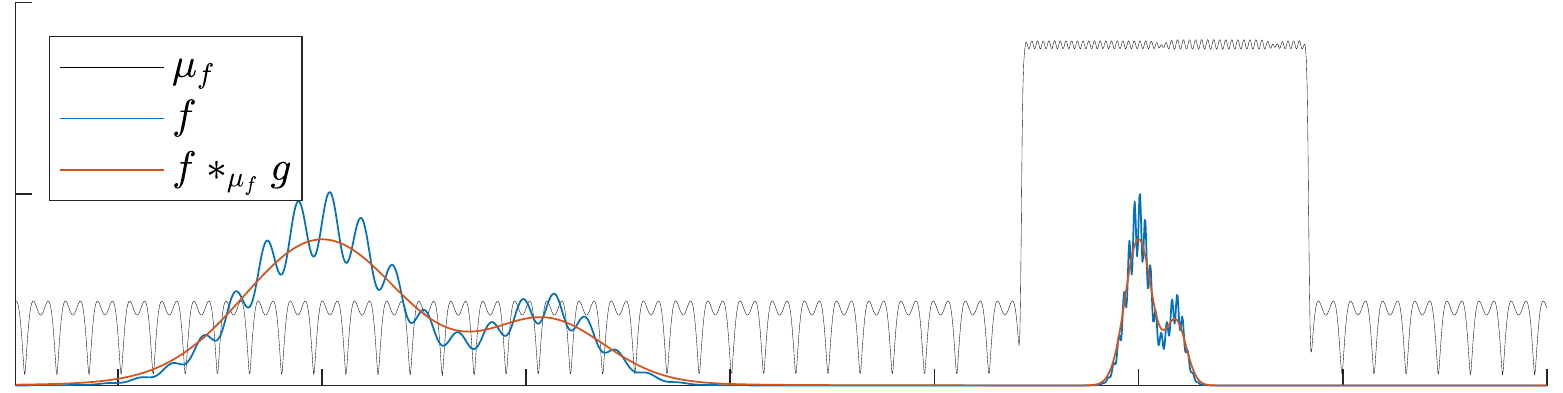}                
	   \end{subfigure}%
       \vspace{0.2cm}
       \begin{subfigure}[b]{0.8\textwidth}
               \centering
               \includegraphics[width=1\textwidth]{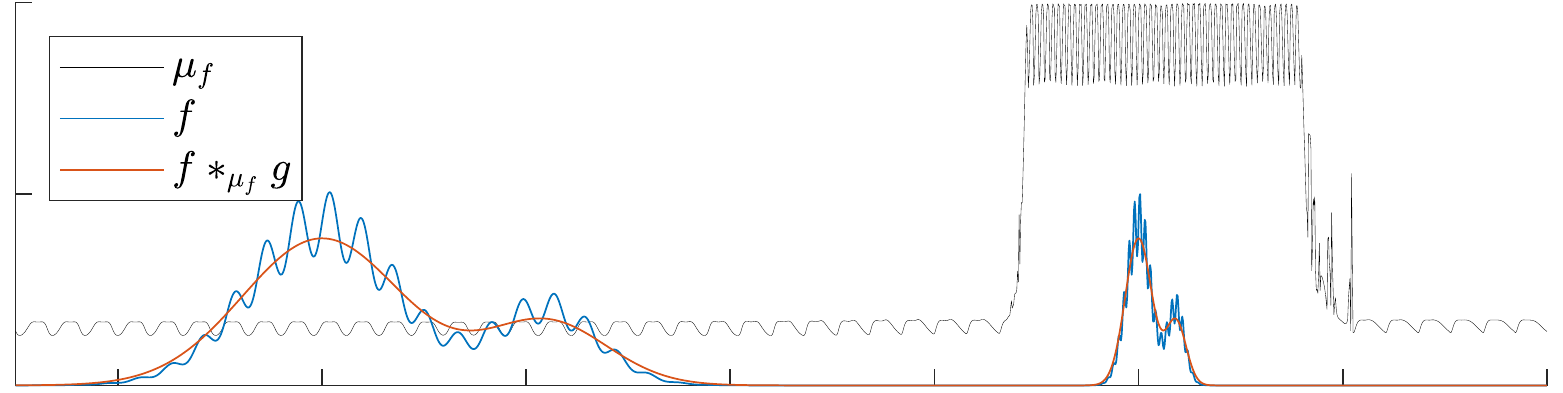}
       \end{subfigure}
        \vspace{0.2cm}
        \begin{subfigure}[b]{0.8\textwidth}
                \centering
                \includegraphics[width=1\textwidth]{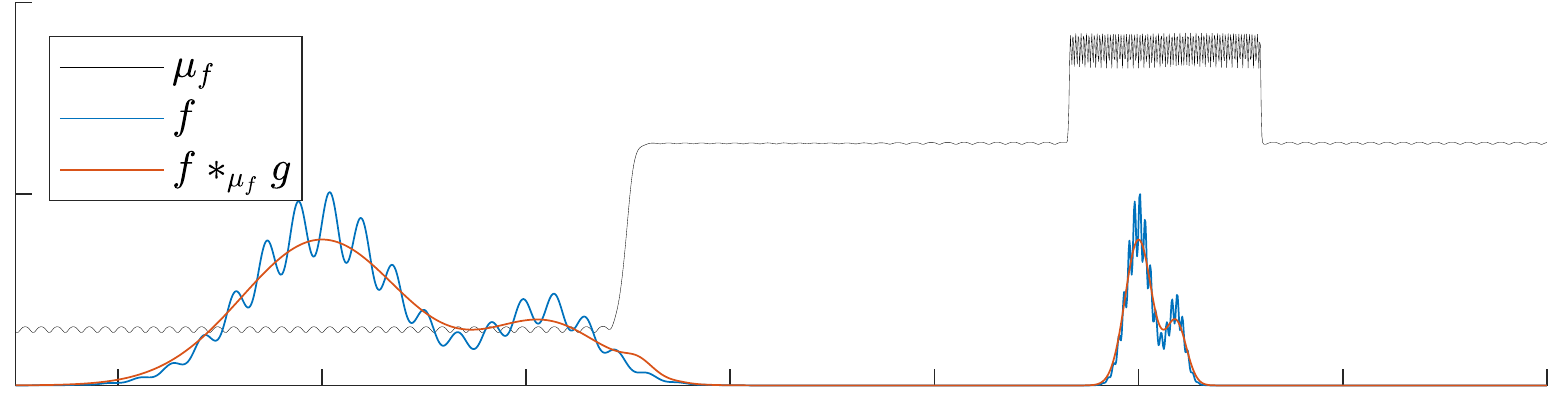}
        \end{subfigure}
        \caption{$\mu_f^{(c)}$, $\mu_f^{(d)}$ and $\mu_f^{(e)}$ as given by the
        formulas \eqref{equ:FBIChoiceMu}, \eqref{equ:adaptiveFBImu} and
        \eqref{equ:finalChoiceMu} describe the local variation of $f$. Choosing them as adaptation functions yields proper local scaling of $g$ and thereby an adequate smoothing of $f$
        everywhere (the width $\sigma$ of the Gaussian kernel $g$ as well as
        $Q$ in \eqref{equ:FBIChoiceMu} and $\lambda$ in
        \eqref{equ:adaptiveFBImu} were chosen manually).}
        \label{fig:convolution6und7und8}
\end{figure}

%

\end{example}

\begin{example}
\label{example:banana}
We also present a 2-dimensional example, where $f$ is chosen as the following
highly curved density (a strongly deformed Gaussian):
\begin{equation}
\label{equ:bananaDensity}
f(x)
=
\frac{1}{2\pi \sigma}\, 
\exp\left(-\frac{1}{2} \left[\Big(\frac{x_1}{\sigma}\Big)^2 + \bigg(x_2 - \alpha\Big(\frac{x_1}{\sigma}\Big)^2\bigg)^2\right]  \right),
\qquad
\alpha = 4,\ \sigma = 5.
\end{equation}
\begin{figure}[H]
        \centering
        \begin{subfigure}[b]{0.24\textwidth}
            \centering
			\includegraphics[width=\textwidth]{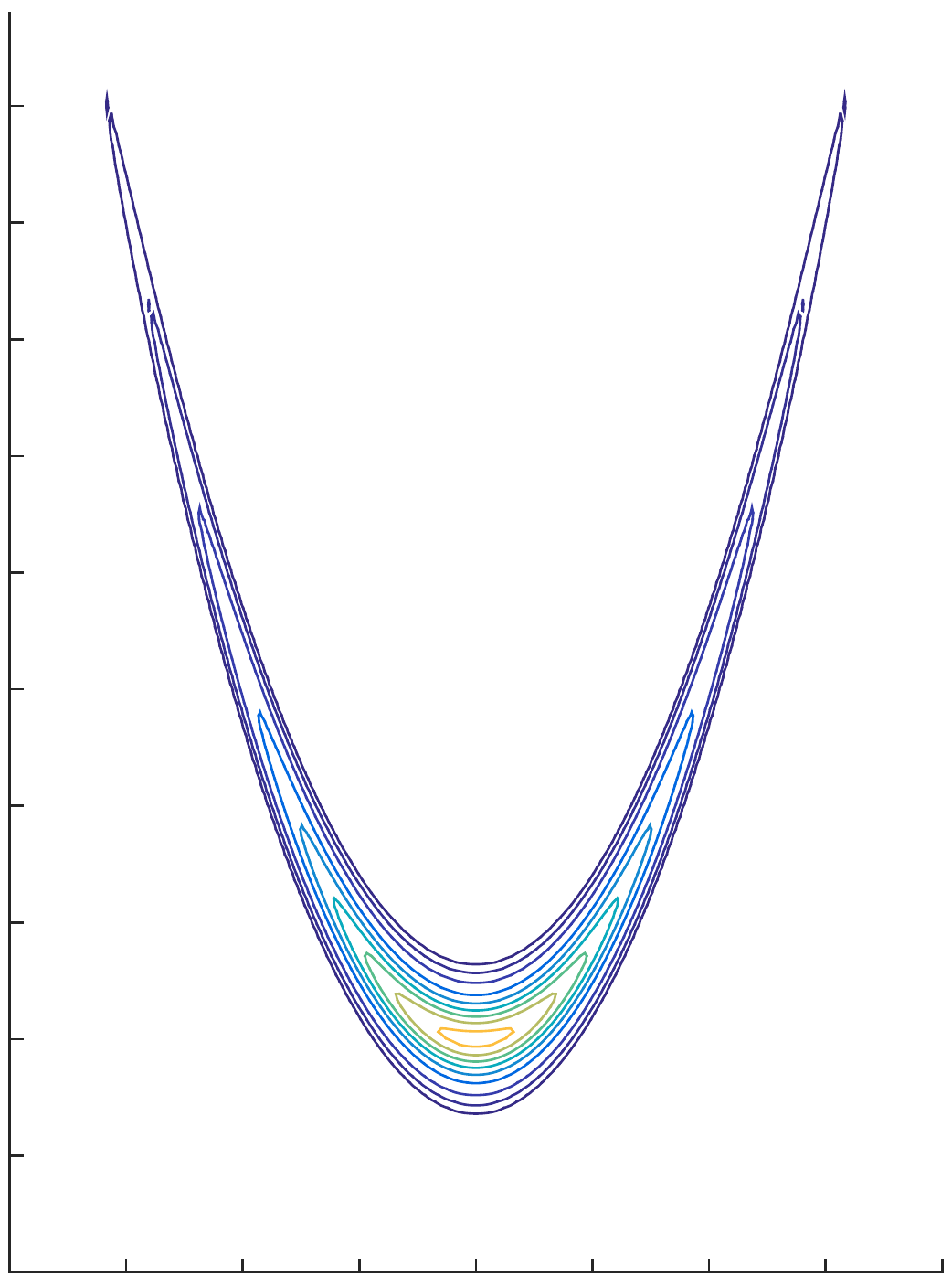}
            \caption{$f$ given by \eqref{equ:bananaDensity}}
        \end{subfigure}
        \hfill
        \begin{subfigure}[b]{0.24\textwidth}
            \centering
			\includegraphics[width=\textwidth]{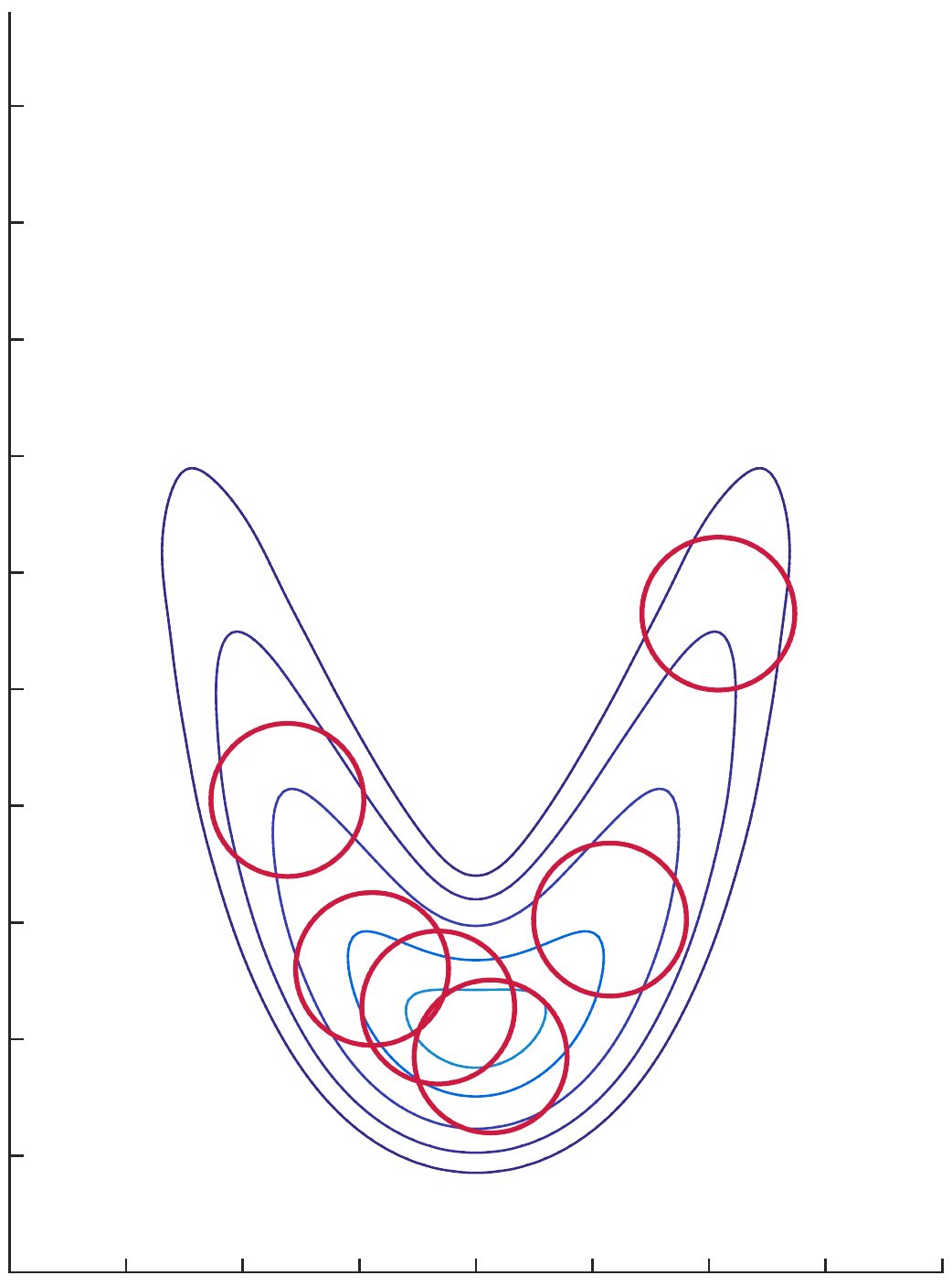}
			\caption{$f\ast g$}
		\end{subfigure}
		\hfill
        \begin{subfigure}[b]{0.24\textwidth}
            \centering
			\includegraphics[width=\textwidth]{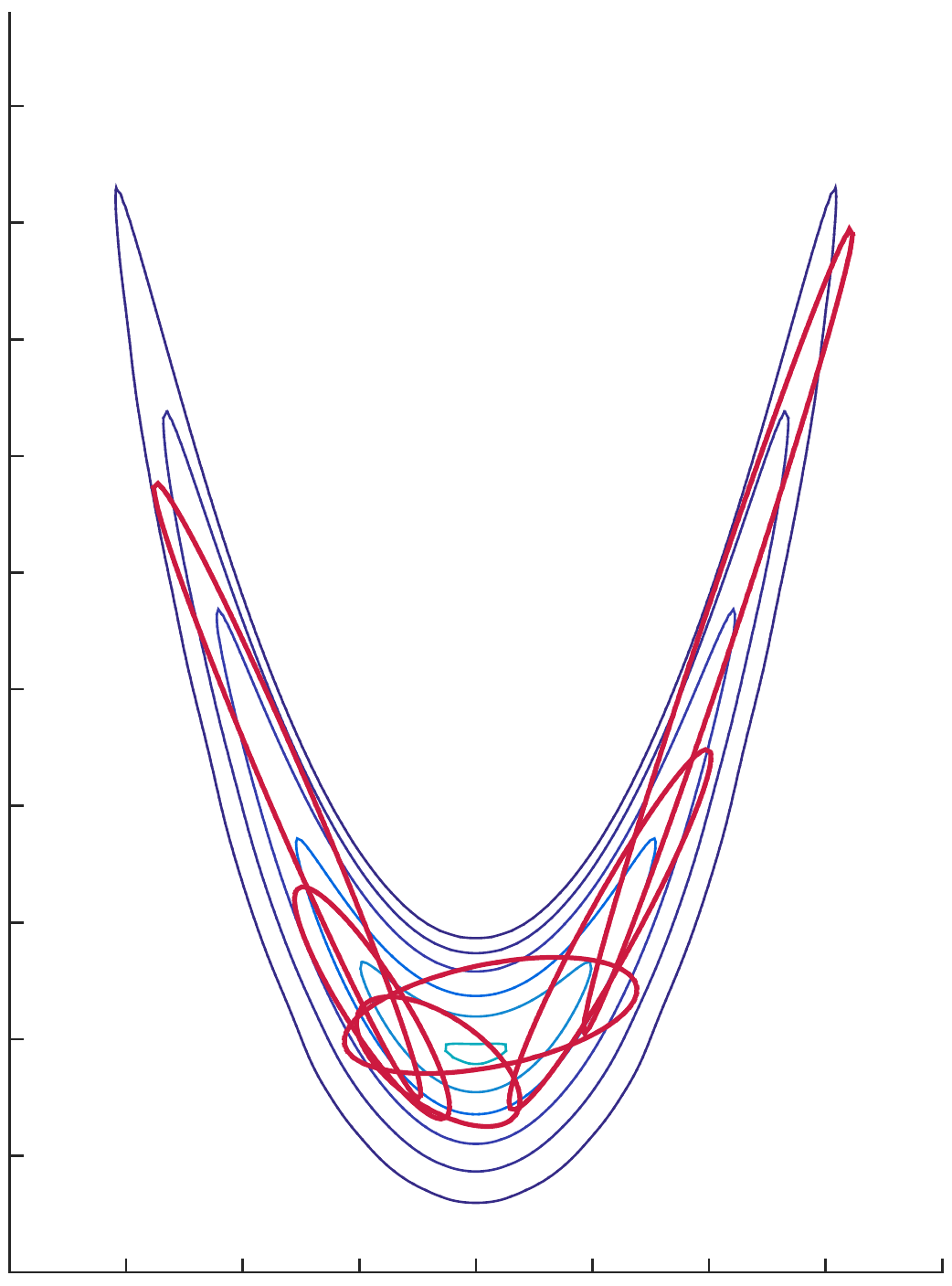}
            \caption{$f\ast_{\mu_f} g$ using \eqref{equ:adaptiveFBImu}}
        \end{subfigure}
        \hfill
        \begin{subfigure}[b]{0.24\textwidth}
            \centering
			\includegraphics[width=\textwidth]{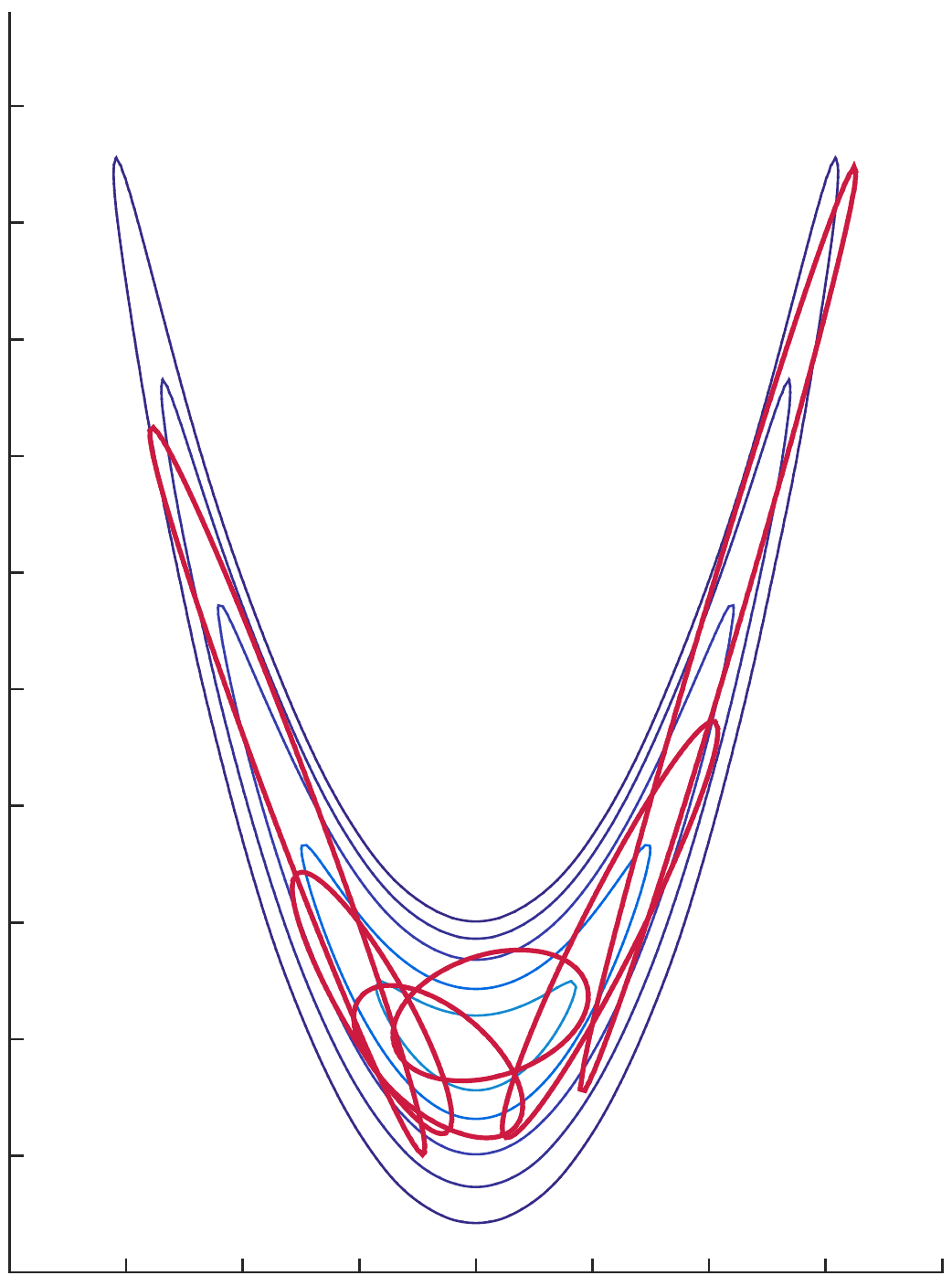}
			\caption{$f\ast_{\mu_f} g$ using \eqref{equ:finalChoiceMu}}
		\end{subfigure}
        \caption{The curved density \eqref{equ:bananaDensity} smoothed by
        standard convolution and adaptive convolution using rules
        \eqref{equ:adaptiveFBImu} and \eqref{equ:finalChoiceMu} (the width $\sigma$ of the Gaussian kernel $g$ as well as $\lambda$ in \eqref{equ:adaptiveFBImu} were chosen manually). It is evident how the last two adapt to the local behavior of $f$. The red ellipses show 80\%
        contours of the kernels $g(\Cdot-y)$ in (b) and the (stretched) kernels
        $\lvert\det\mu_f(y)\rvert\, g\big(\mu_f(y)(\Cdot-y)\big)$ in (c) and
        (d) for several centers $y$.
        }
        \label{fig:banana_density}
\end{figure}
\end{example}

Our last example demonstrates, how the Adaptation Axiom
\ref{cond:adaptation}(A4) is fulfilled by \eqref{equ:adaptiveFBImu} but
violated by \eqref{equ:finalChoiceMu}. The latter fails to capture solely
local properties of $f$ because it makes use of the convolution of $f^2$ with
itself (and with its derivatives).

\begin{example}
\label{example:threeGaussians}
Consider the function
\[
f(x) = f_1(x) + f_2(x-ta) + f_3(x+ta),
\]
where $f_1,f_2,f_3$ are three functions `located at the origin', $t>0$ and
$a\in\R^d\setminus \{0\}$. For $t\to\infty$ the three parts will `drift apart'.
However, the convolution of $f_2(\Cdot-ta)$ and $f_3(\Cdot+ta)$ (and of their
derivatives) remain unchanged as $t$ grows and, since \eqref{equ:finalChoiceMu}
depends on these convolutions, $f_2$ and $f_3$ will have an influence on the
smoothing of $f_1$ no matter how large $t$ becomes.
\begin{figure}[H]
        \centering
        \begin{subfigure}[b]{0.32\textwidth}
            \centering
			\includegraphics[width=\textwidth]{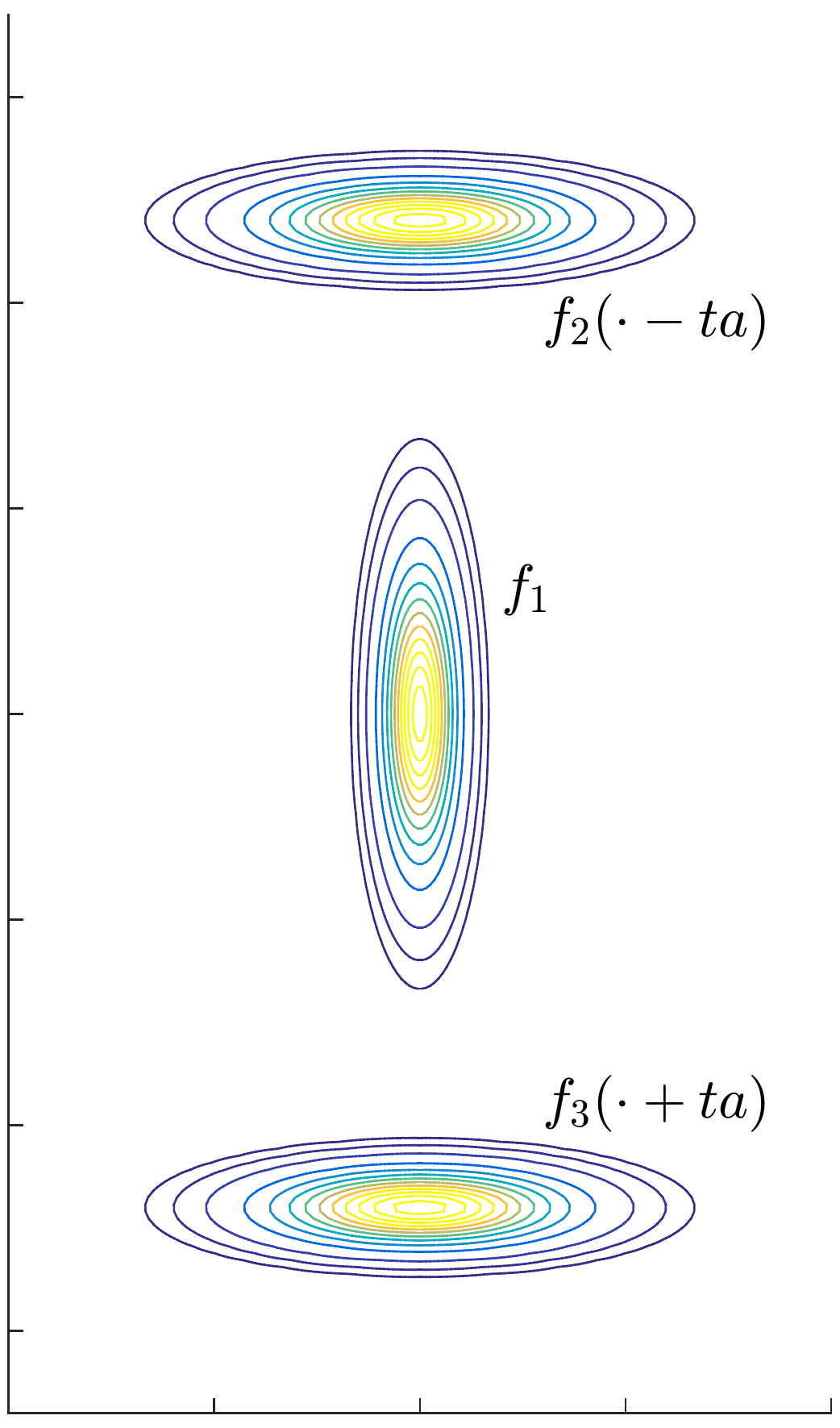}
            \caption{$f$ from Example
            \ref{example:threeGaussians}}
        \end{subfigure}
        \hfill
        \begin{subfigure}[b]{0.32\textwidth}
            \centering
			\includegraphics[width=\textwidth]{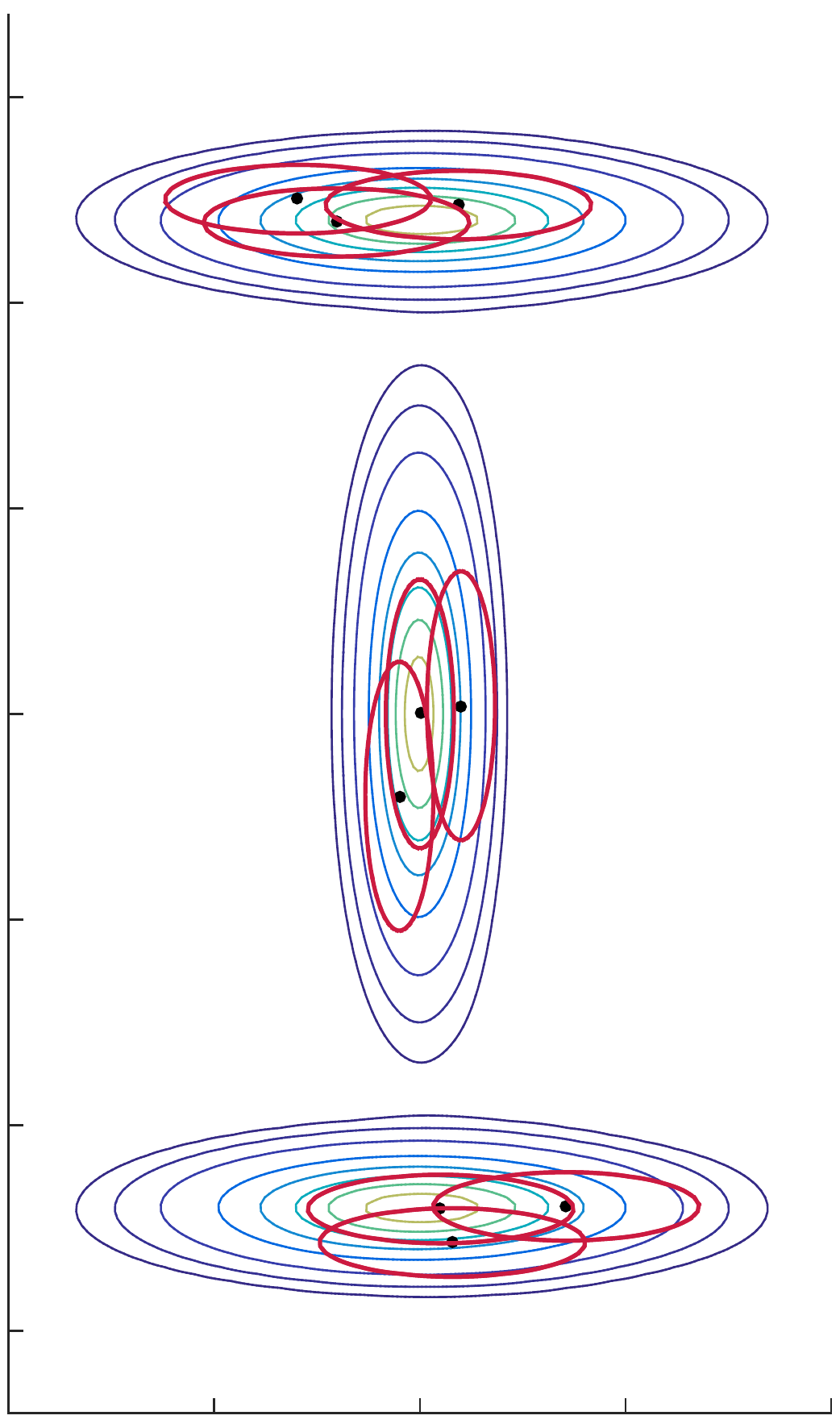}
			\caption{$f\ast_{\mu_f} g$ using \eqref{equ:adaptiveFBImu} for $\mu_f$}
		\end{subfigure}
		\hfill
        \begin{subfigure}[b]{0.32\textwidth}
            \centering
			\includegraphics[width=\textwidth]{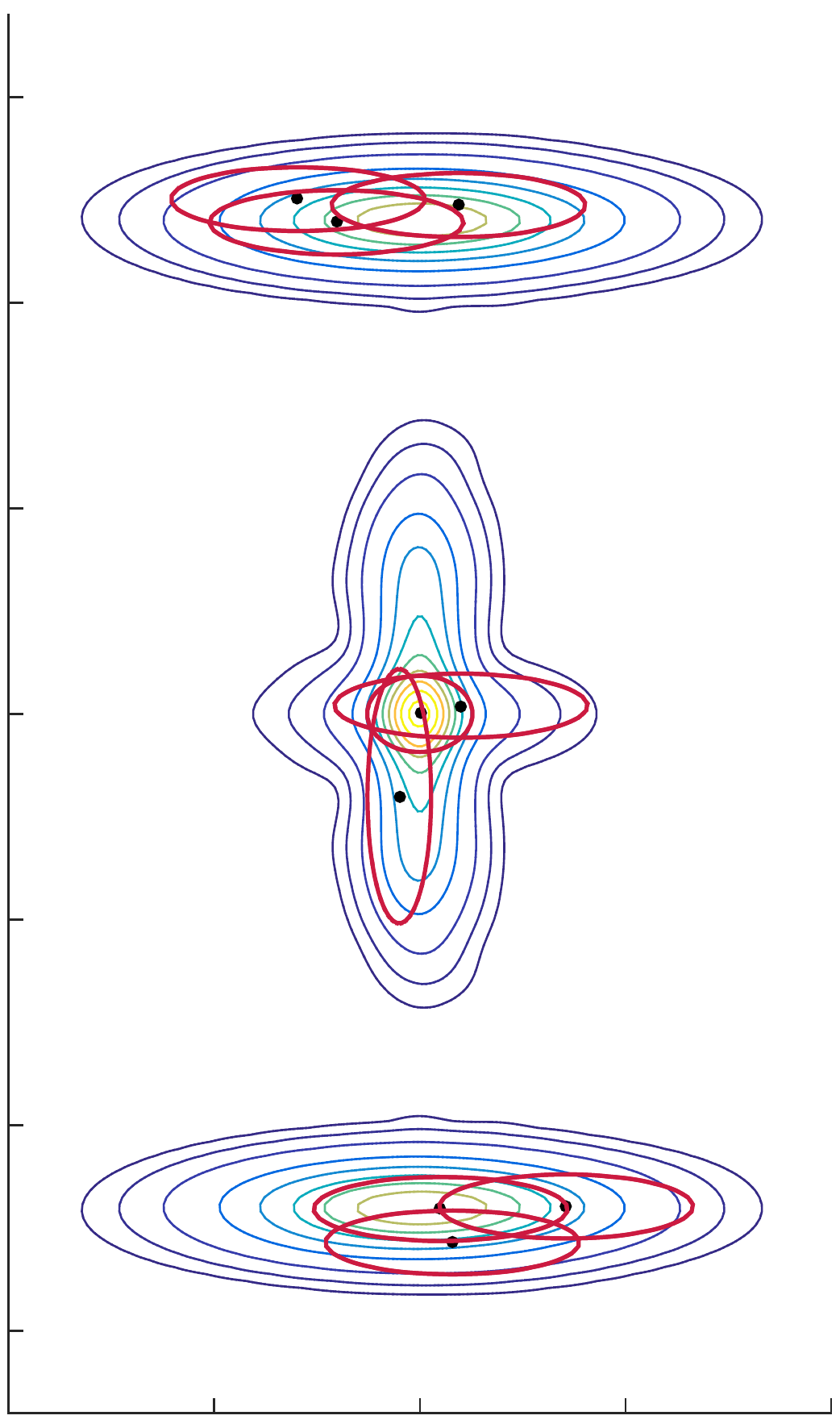}
			\caption{$f\ast_{\mu_f} g$ using \eqref{equ:finalChoiceMu} for $\mu_f$}
		\end{subfigure}
        \caption{Adaptive smoothing of the function $f$ from Example
        \ref{example:threeGaussians} using the rules \eqref{equ:adaptiveFBImu}
        and \eqref{equ:finalChoiceMu}.
        While the two Gaussians $f_2(\Cdot-ta)$ and $f_3(\Cdot+ta)$ are smoothed
        nicely in both cases (in fact, $\mu_f$ is constant there as stated in
        Corollary \ref{cor:gaugeWidth2}), the smoothing of $f_1$ via 
        \eqref{equ:finalChoiceMu} is highly influenced by the other two parts, resulting in undesirable
        oversmoothing in $x_2$-direction at its center. This effect remains
        unchanged even if $t$ is increased. The red ellipses show 80\% contours of the (stretched) kernels $\lvert\det\mu_f(y)\rvert\, g\big(\mu_f(y)(\Cdot-y)\big)$ for several centers $y$. The width
        $\sigma$ of the Gaussian kernel $g$ as well as $\lambda$ in
        \eqref{equ:adaptiveFBImu} were chosen manually.
        }
        \label{fig:three_gauss_density}
\end{figure}
\end{example}

%% file: sections/Conclusion.tex
\section{Conclusion}

After defining adaptive convolutions (for other types of adaptive convolutions
see Appendix \ref{section:Convolution2}) and analyzing their theoretical
properties, we have derived a formula for the adaptation function
$\mu_f$, which allows automatic adjustment of the local smoothing of a function $f$.
The requirements for such a formula were reasonable axioms on how the
adaptation function $\mu_f$ should behave under transformations (shifting and
scaling) of $f$.
Its derivation relied on the notion of the local variation of $f$, which we
argued can be quantified by means of certain phase space transforms.
Several suggestions for the mapping $f\mapsto\mu_f$ were made, but only
\eqref{equ:adaptiveFBImu} succeeds in fulfilling all Adaptation Axioms
\ref{cond:adaptation}.

The choice \eqref{equ:finalChoiceMu} looked promising, but failed to capture
solely local properties of $f$ and therefore, as demonstrated in Example
\ref{example:threeGaussians}, could not realize a key property required for
adaptive convolutions:
If the function $f = \sum_{k=1}^K f_k$ is the sum of several well-separated
functions $f_1,\dots,f_K$, then its adaptive convolution should also be
(approximately) the sum of the adaptive convolutions of $f_1,\dots,f_K$
(Theorem \ref{theorem:AdaptationConditions}(iv)).

The overall extent of the smoothing effect can be calibrated by choosing the width $\sigma$ of the smoothing kernel $g$ and applying the adaptive convolution to Gaussian functions $G[a,\Sigma]$, see Corollary \ref{cor:gaugeWidth2}. The choice of the parameter $\lambda$ remains an open problem. In our examples the values in the interval $[0.6,1.2]$ provided favorable results.

As a byproduct, we obtain an adaptive window selection method for time-frequency
representations, which is invariant under linear transformations of the signal,
allows different window sizes in different directions and adapts locally to the
signal's (mean squared) frequency.

The considerations in this paper are mainly theoretical and the computation of
the adaptation functions \eqref{equ:adaptiveFBImu} and \eqref{equ:finalChoiceMu}
appears intricate and costly (except for simple examples like the ones
presented here).
An application of adaptive convolutions to variable kernel density estimation
is discussed in a companion paper \cite{klebanovVKDE}, where also a numerical scheme for the
computation was developed in the case of Gaussian kernels.

%% file: sections/OtherTypes.tex
\section{Other Types of Adaptive Convolutions}
\label{section:Convolution2}

In the case of the common convolution $f\ast g$, the contribution of $f(y)$ to
$(f\ast g) (x)$ depends, roughly speaking, on the distance between $x$ and $y$.
In the following, we will introduce two further types of adaptive
convolutions, for which the contribution of $f(y)$ to the convolution evaluated
in $x$ depends on the distance between $h(x)$ and $h(y)$, where $h$ is a
function which controls the adaptation.

\begin{definition}[adaptive convolutions of types two and three]
\label{def:weightedAdapt}
Let $1\le p\le \infty$, $f\in L^1\left(\R^d\right),\ g_1\in
L^1\left(\R^n\right),\ h\colon \R^d\to\R^n$ be
a measurable function and $g_2,g\in L^p\left(\R^n\right)$ such that
\[
0 <
\left\|g(h(\Cdot)-z)\right\|_p<\infty\quad\text{and}\quad 0 <
\left\|g_2(h(\Cdot)-z)\right\|_p < \infty\qquad \text{for almost all
$z\in\R^d$}.
\]
\\
We define the \emph{$h$-adaptive convolutions of types two and three} by
\begin{align*}
(f\ast^p[g\, |\, h])(x)
&=
f\bar\ast G_p\, ,
\qquad
(f\ast^p[g_1,g_2\, |\, h])(x)
=
f\bar\ast \tilde G_p \, ,
\intertext{where}
G_p(x,y)
&=
\left\|g\right\|_p
\frac{g\left(h(x)-h(y)\right)}{\left\|g(h(\Cdot)-h(y))\right\|_p}\, ,
\\
\tilde G_p(x,y)
&=
\left\|g_2\right\|_p\int g_1(z-h(y))\,
\frac{g_2(z-h(x))}{\left\|g_2(z-h(\Cdot))\right\|_p} \, \mathrm dz\, .
\end{align*}
Again, we will omit the index $p$ in the case $p=1$.
\end{definition}

\begin{remark}
$f\ast^p[g\, |\, h]$ is homogeneous in $g$ and $f\ast^p[g_1,g_2\, |\, h]$
is homogeneous in $g_2$ and linear in $g_1$.
\end{remark}

\begin{proposition}[Young's inequality]
\label{prop:young2}
Under the conditions of Definition \ref{def:weightedAdapt}, we have:
\[
\left\|f\ast^p [g\, |\, h]\right\|_p\le
\left\|f\right\|_1\left\|g\right\|_p
\qquad\text{and}\qquad
\left\|f\ast^p [g_1,g_2\, |\, h]\right\|_p\le
\left\|f\right\|_1\left\|g_1\right\|_1 \left\|g_2\right\|_p\, .
\]
\end{proposition}

\begin{figure}[H]
\centering
\begin{subfigure}[b]{0.25\textwidth}
	\centering
	\includegraphics[width=\textwidth]{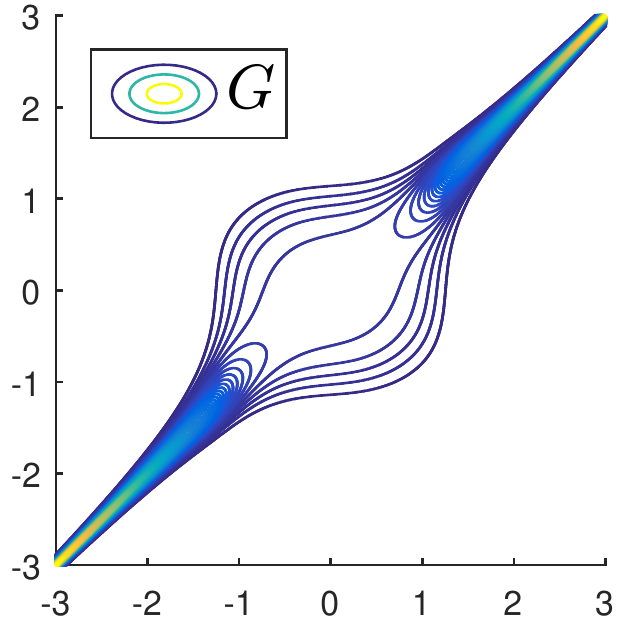}	
\end{subfigure}
\hspace{0.4cm}
\begin{subfigure}[b]{0.5\textwidth}
	\centering
	\includegraphics[width=\textwidth]{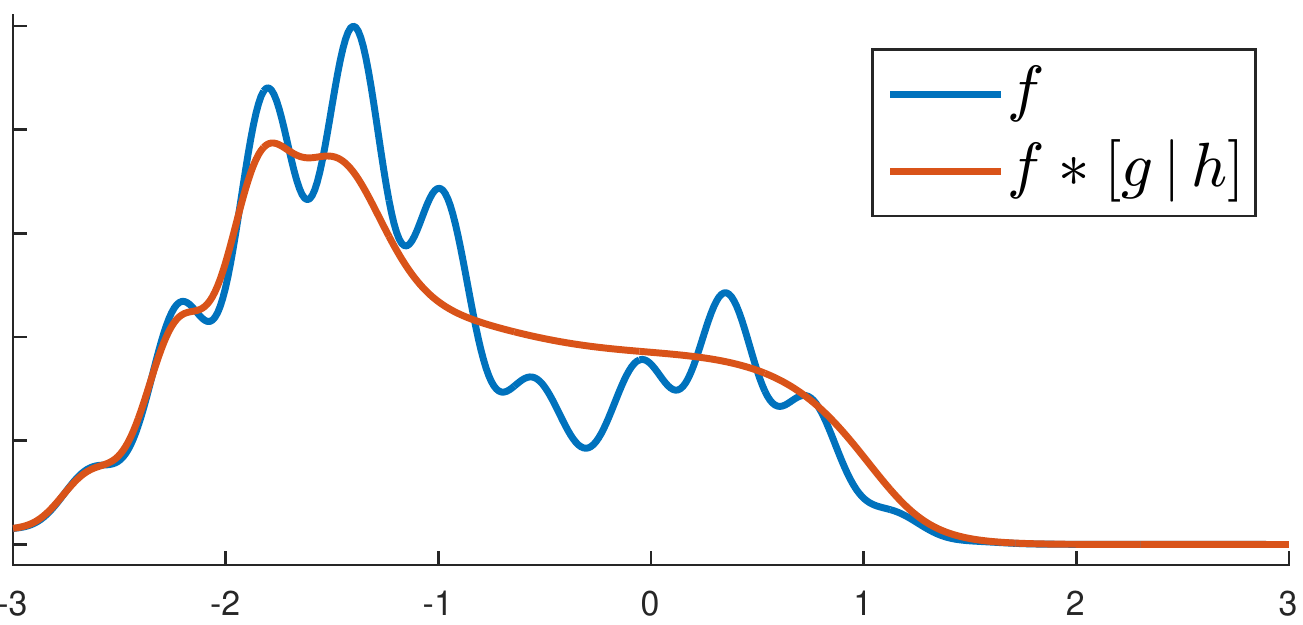}	
\end{subfigure}        
        \vfill
\begin{subfigure}[b]{0.25\textwidth}
	\centering
	\includegraphics[width=\textwidth]{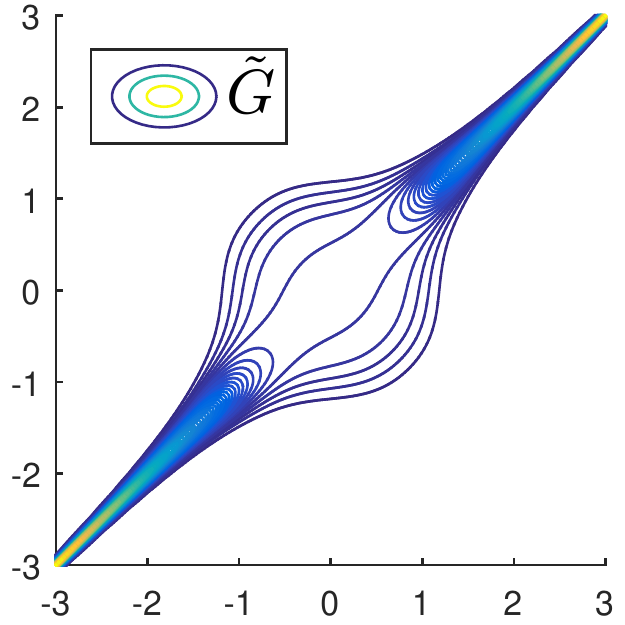}	
\end{subfigure}
\hspace{0.4cm}
\begin{subfigure}[b]{0.5\textwidth}
	\centering
	\includegraphics[width=\textwidth]{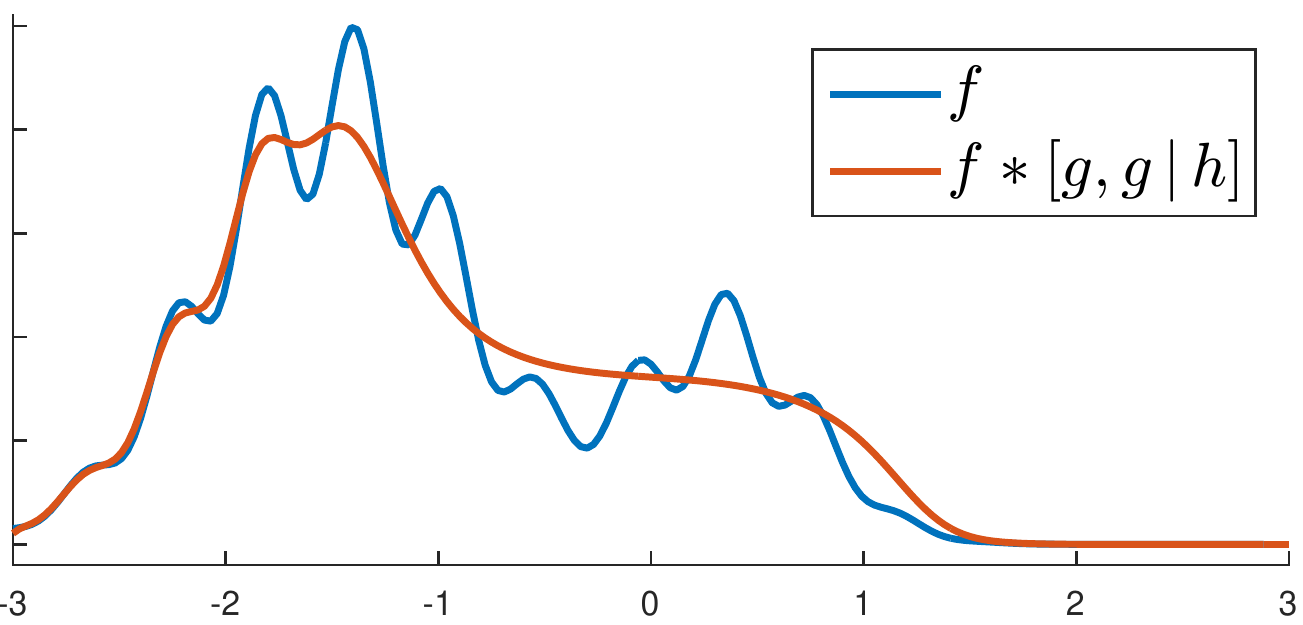}	
\end{subfigure}
\caption{Adaptive Convolutions of type two and three for Gaussian $g,\, g_1,\,
g_2$, a cubic polynomial $h(x) = x^3 + x/3$ and $p=1$. Note that $\tilde
G$ is symmetric, while $G$ is not (see also Proposition
\ref{prop:symmetricG}). Both convolutions provide a strong smoothing close to
zero, and nearly no smoothing away from zero, where $G$ and $\tilde G$ act
nearly like Dirac $\delta$-distributions.}
\label{fig:weighted_cubic}
\end{figure}


\begin{figure}[H]
\centering
\begin{subfigure}[b]{0.25\textwidth}
	\centering
	\includegraphics[width=\textwidth]{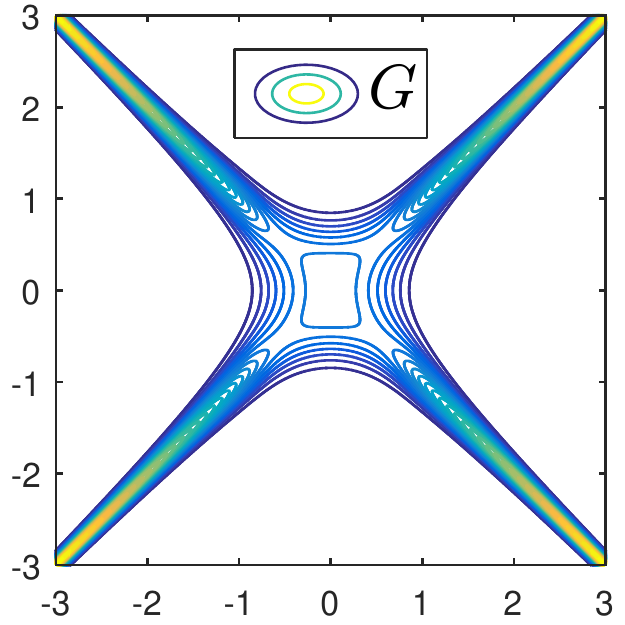}
\end{subfigure}
\hspace{0.4cm}
\begin{subfigure}[b]{0.5\textwidth}
	\centering
	\includegraphics[width=\textwidth]{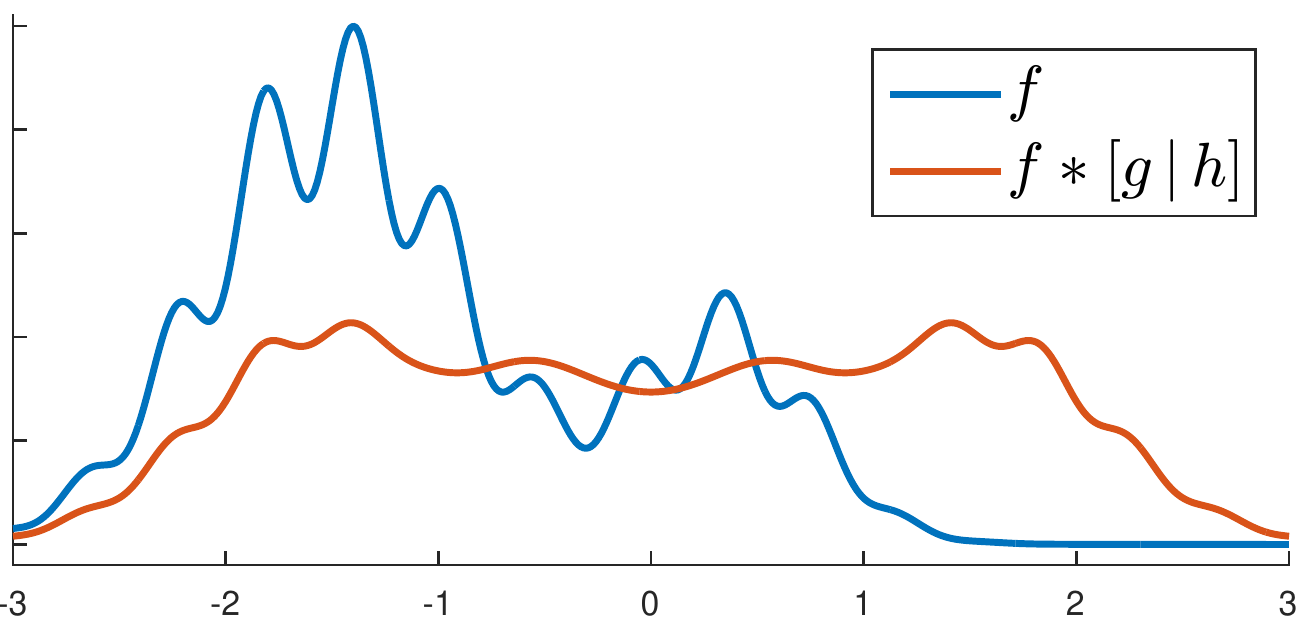}
\end{subfigure}        
        \vfill
\begin{subfigure}[b]{0.25\textwidth}
	\centering
	\includegraphics[width=\textwidth]{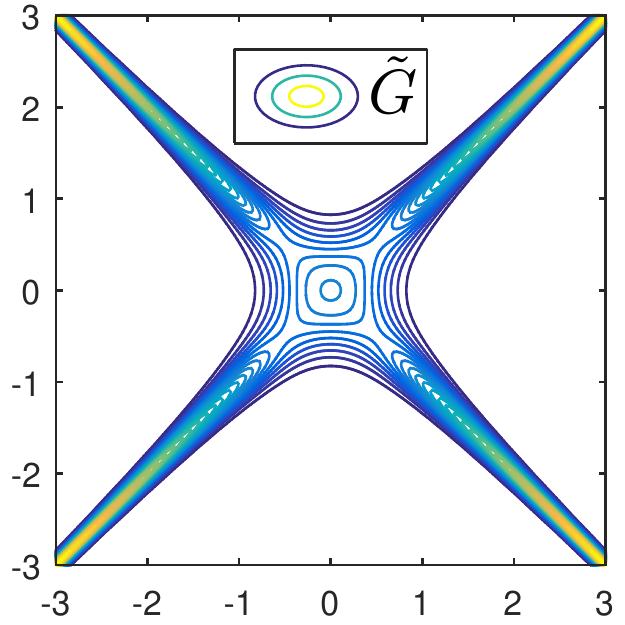}
\end{subfigure}
\hspace{0.4cm}
\begin{subfigure}[b]{0.5\textwidth}
	\centering
	\includegraphics[width=\textwidth]{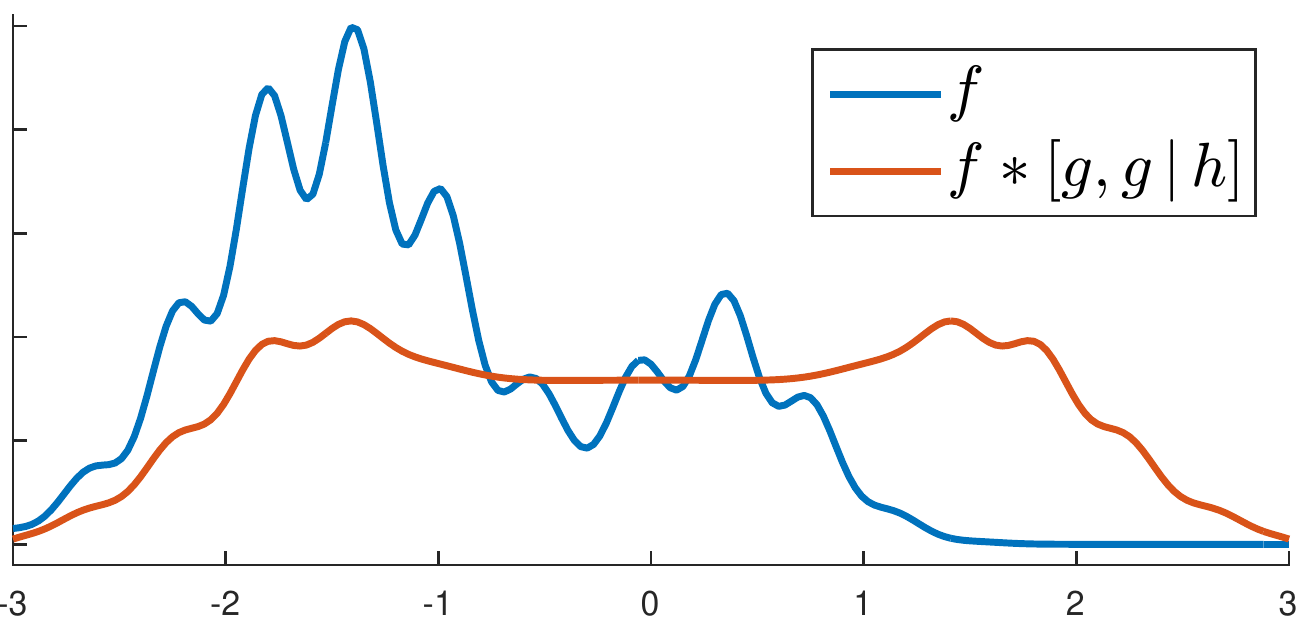}	
\end{subfigure}
\caption{Adaptive Convolutions of type two and three for Gaussian $g,\, g_1,\,
g_2$, a quadratic $h(x) = x^2$ and $p=1$. Observe how the function $f$ is
`smoothed over to the right side' due to the non-injectivity of $h$. Again,
the smoothing is stronger close to zero.}
\label{fig:weighted_quadratic}
\end{figure}

If $g_1 = g_2 =: g$, the
$h$-adaptive convolution of type three fulfills the general Young's
inequality:
\begin{proposition}
\label{prop:symmetricG}
Let $1\le p,q,r\le\infty$ such that $1+\frac{1}{r} = \frac{1}{p} +
\frac{1}{q}$.
Let $f\in L^q\left(\R^d\right)$, $g_1=g_2=:g\in L^1\cap
L^p\left(\R^n\right)$ and $h\colon \R^d\to\R^n$ be a measurable function such that
\\
$0 < \left\|g(h(\Cdot)-z)\right\|_p < \infty$ for almost all
$z\in\R^d$. Then $\tilde G_p$ is symmetric and
\[
\left\|f\ast^p [g,g\, |\, h]\right\|_r\le
\left\|f\right\|_q\left\|g\right\|_1 \left\|g\right\|_p\, .
\]
\end{proposition}

\begin{remark}
Several generalizations to adaptive convolutions of types two and three are
possible:
\begin{enumerate}
 \item
Nothing changes, if we replace $h$ by two different functions $h_1,\, h_2\colon \R^d\to\R^n$ in the following way:
 \begin{align*}
(f\ast^p[g\, |\, h_1,h_2])(x)
&=
f\bar\ast G_p\, ,
\qquad
(f\ast^p[g_1,g_2\, |\, h_1,h_2])(x)
=
f\bar\ast \tilde G_p\, ,
\intertext{where}
G_p(x,y)
&=
\left\|g\right\|_p
\frac{g\left(h_1(x)-h_2(y)\right)}{\left\|g(h_1(\Cdot)-h_2(y))\right\|_p}\, ,
\\
\tilde G_p(x,y)
&=
\left\|g_2\right\|_p\int g_1(z-h_1(y))\,
\frac{g_2(z-h_2(x))}{\left\|g_2(z-h_2(\Cdot))\right\|_p} \, \mathrm dz\, .
\end{align*}
 \item
If
$g_1\in
L^1\left(\R^d\times\R^n\right),\ g_2,\, g\in L^p\left(\R^d\times\R^n\right)$
depend on an additional parameter in $\R^d$ and
\begin{itemize}
  \item $\displaystyle\int_{\R^d} |g_1(y,z-h(y))|\, \mathrm dy\le \Gamma_1$ for
  some constant $\Gamma_1>0$ (independent of $z$),
\item for almost all $z\in\R^d$,
\[
0< \int |g(\tilde x, z-h(\tilde x))|^p\, \mathrm d\tilde x < \infty
\ \quad\text{and}\quad\ 
0< \int |g_2(\tilde x, z-h(\tilde x))|^p\, \mathrm d\tilde x < \infty,
\]
\end{itemize}
we can generalize
\begin{align*}
G_p(x,y)
&=
\left\|g\right\|_p\, 
\frac{g\left(x,h(x)-h(y)\right)}{\left(\int |g(\tilde x, z-h(\tilde x))|^p\,
\mathrm d\tilde x\right)^{1/p}}\, ,
\\
\tilde G_p(x,y)
&=
\left\|g_2\right\|_p\, \int g_1(y,z-h(y))\,
\frac{g_2(x,z-h(x))}{\left(\int |g_2(\tilde x, z-h(\tilde x))|^p\, \mathrm
d\tilde x\right)^{1/p}} \, \mathrm dz\, .
\end{align*}
In this case, Young's inequality takes the forms
\[
\left\|f\ast^p [g\, |\, h]\right\|_p\le
\left\|f\right\|_1\left\|g\right\|_p
\qquad\text{and}\qquad
\left\|f\ast^p [g_1,g_2\, |\, h]\right\|_p\le
\Gamma_1\, \left\|f\right\|_1 \left\|g_2\right\|_p\, .
\]
\end{enumerate}
\end{remark}
\begin{remark}
If $g,g_1,g_2$ are Gaussian functions or similar (in the sense that they
attain their maximum in the origin and decay monotonically to zero as
$x\to\pm\infty$) and $D h(x)$ is invertible
for each $x\in\R^d$, both convolutions approximate the $\mu$-adaptive convolution
from Section \ref{section:Theory} for $\mu = D h$.
More precisely, in this case the
linearization of $h$ at $y$, $h(x)-h(y)\approx Dh(y)\, (x-y)$, is meaningful and yields
\begin{align*}
G_p(x,y)
&=
\left\|g\right\|_p
\frac{g\left(h(x)-h(y)\right)}{\left\|g(h(\Cdot)-h(y))\right\|_p}
\approx
\left\|g\right\|_p \frac{g\left(Dh(y)\, (x-y)\right)}{\left\|g(Dh(y)\,
(\Cdot-y))\right\|_p}
=
|\det Dh(y)|^{1/p} g\left(Dh(y)\, (x-y) \right),
\\[0.2cm]
\tilde G_p(x,y)
&=
\left\|g_2\right\|_p \int g_1(z)\,
\frac{g_2(z-h(x)+h(y))}{\left\|g_2(z-h(\Cdot)+h(y))\right\|_p} \, \mathrm dz
\approx
\left\|g_2\right\|_p \int g_1(z)\,
\frac{\tau g_2(Dh(y)\, (x-y)-z)}{\left\|g_2(Dh(y)\, (\Cdot-y)-z)\right\|_p} \,
\mathrm dz
\\
&=
|\det Dh(y)|^{1/p}\, (g_1\ast \tau g_2)\left(Dh(y)\, (x-y) \right),
\end{align*}
where $\tau g_2 (x):= g_2(-x)$ denotes the reflection of $g_2$.
This implies
\[
f\ast^p[g\, |\, h]
\approx
f\ast_{Dh}^pg
\qquad
\text{and}
\qquad
f\ast^p[g_1, g_2\, |\, h]
\approx
f\ast_{Dh}^p(g_1\ast \tau g_2).
\]
This observation is clear on an intuitive level, since $G_p(x,y)$ is the
magnitude of the contribution of $f(y)$ in the term $\left(f\ast^p [g\, |\, h]\right)(x)$. In the case of the
common convolution (and Gaussian $g$ or similar), this magnitude depends on the
distance between $x$ and $y$. Here, it depends on the distance between $h(x)$
and $h(y)$, hence the convolution is weighted by the `slope' of $h$, see
Figure \ref{fig:weighted_cubic}.
\vspace{0.2cm}

However, if $h$ is not injective, the new convolutions $f\ast^p[g\,
|\, h]$ and $f\ast^p[g_1, g_2\, |\, h]$ yield further possibilities. For
example, we can `let a value $f(y)$ contribute strongly to $f\ast^p[g\,
|\, h](x)$, even though $x$ is far away from $y$ (without contributing strongly to most values in
between)' by choosing $h$ such that $h(y)\approx h(x)$, see Figure
\ref{fig:weighted_quadratic}.
\end{remark}

%% file: sections/Proofs.tex
\section{Proofs}
\label{section:proofs}

\begin{proof}[Proof of Theorem \ref{theorem:young1}]
The cases $p=1$ and $p = \infty$ are straightforward:
\begin{align*}
&\|f\bar\ast G\|_1
\le 
\int_{\R^d} \int_{\R^d} |f(y)|\, |G(x,y)|\, \mathrm dy \, \mathrm dx
=
\int_{\R^d} |f(y)| \int_{\R^d} |G(x,y)|\, \mathrm dx \, \mathrm dy 
\le
\left\|f\right\|_1 \Gamma \, ,
\\[0.2cm]
&\|f\bar\ast G\|_\infty
\le 
\operatorname*{ess\,sup}_{x\in\R^d} \int_{\R^d} |f(y)|\,
\left| G(x,y)\right|\, \mathrm dy
\le 
\int_{\R^d} |f(y)|\, \mathrm dy\, \Gamma
=
\|f\|_1\, \Gamma\, .
\end{align*}
Now, let $1<p<\infty$ and $p'$ denote its conjugate
exponent, i.e.
$\frac{1}{p}+\frac{1}{p'} = 1$.
H\"older's inequality yields
\begin{align*}
|f \bar\ast G|(x)
&\le 
\int_{\R^d} |f(y)|^{\frac{1}{p'}}\, |f(y)|^{\frac{1}{p}}\,
|G(x,y)|\, \mathrm dy
\le 
\left\||f|^{\frac{1}{p'}}\right\|_{p'}\, \left\||f|^{\frac{1}{p}}\, 
|G(x,\Cdot)|\right\|_{p},
\end{align*}
which implies
\begin{align*}
&&&\|f\bar\ast G\|_{p}^p&
&\le &
&\left\||f|^{\frac{1}{p'}}\right\|_{p'}^p\, 
\int_{\R^d} \left\||f|^{\frac{1}{p}}\, 
|G(x,\Cdot)|\right\|_{p}^p \, \mathrm dx&
&=&
&\big\|f\big\|_1^{p/p'}\, 
\int_{\R^d} \int_{\R^d}|f(y)|\, |G(x,y)|^p \ \mathrm dy\ \mathrm dx&&&
\\
&&&&&=&
&\big\|f\big\|_1^{p/p'}\, 
\int_{\R^d}|f(y)| \left\|G(\Cdot,y)\right\|_p^p \, \mathrm dy&
&\le &
&\big\|f\big\|_1^{1+p/p'}\, 
\Gamma^p
\le 
\|f\|_{1}^p\, \Gamma^p\, .&&&
\end{align*}
\end{proof}

\begin{proof}[Proof of Corollary \ref{cor:young}]
First note that for $p<\infty$ and $y\in\R^d$ the change of variables formula
implies:
\begin{equation}
\label{equ:gMuNorm}
\left\| g_{\mu,p}(\Cdot,y)\right\|_p^p
=
\int_{\R^d}|\det\mu(y)|\, \left| g\big(\mu(y)(x-y)\big)\right|^p\,
\mathrm dx
=
\int_{\R^d} \left| g\left(x\right)\right|^p\,
\mathrm dx
=
\left\| g\right\|_p^p.
\end{equation}
For $p=\infty$, the statement $\left\| g_{\mu,p}(\Cdot,y)\right\|_p\le
\left\|g\right\|_p$ is trivial. Theorem \ref{theorem:young1} proves the claim.
\end{proof}

\begin{proof}[Proof of Theorem \ref{theorem:young2}]
The case $r=\infty$ was treated in Theorem \ref{theorem:young1} and
$r<\infty$ implies $p,q<\infty$.
Since
\[
\frac{r-q}{qr} + \frac{r-p}{pr} + \frac{1}{r}
=
\frac{1}{q} - \frac{1}{r} + \frac{1}{p} - \frac{1}{r} + \frac{1}{r}
=
\frac{1}{q} + \frac{1}{p} - \frac{1}{r}
=
1,
\]
the generalized H\"older's inequality yields
\begin{align*}
|f \bar\ast G|(x)
&\le 
\int_{\R^d} |f(y)|^{1-\frac{q}{r}}\, |f(y)|^{\frac{q}{r}}\, 
|G(x,y)|^{1-\frac{p}{r}}\, |G(x,y)|^{\frac{p}{r}}\, \mathrm dy
\\
&\le
\underbrace{\left\||f|^{\frac{r-q}{r}}\right\|_{\frac{qr}{r-q}}}_{=:A}
\underbrace{\left\||G(x,\Cdot)|^{\frac{r-p}{r}}\right\|_{\frac{pr}{r-p}}}_{=:B(x)}
\underbrace{\left\| \left(|f|^q\,|G(x,\Cdot)|^p
\right)^{\frac{1}{r}}\right\|_{r\phantom{\frac{q}{q}}}}_{=:C(x)}
\intertext{and}
A^r
&=
\left(\int |f(y)|^q\, \mathrm dy\right)^{\frac{r-q}{q}}
=
\left\|f\right\|_q^{r-q},
\\
B(x)^r
&=
\left(\int |G(x,y)|^p\, \mathrm dy\right)^{\frac{r-p}{p}}
=
\left\|G(x,\Cdot)\right\|_p^{r-p}
\le \Gamma^{r-p}\, ,
\\
C(x)^r
&=
\int |f(y)|^q\, |G(x,y)|^p\, \mathrm dy\, .
\end{align*}
This implies
\begin{align*}
\|f\bar\ast G\|_{r}^r
&\le
\int A^r\, B(x)^r\, C(x)^r\, \mathrm dx
\le 
\left\|f\right\|_q^{r-q}\, \Gamma^{r-p}
\int |f(y)|^q \underbrace{\int |G(x,y)|^p\, \mathrm dx}_{\le \Gamma^p}\, \mathrm
dy
\le 
\left\|f\right\|_q^{r}\, \Gamma^{r}\, .
\end{align*}
\end{proof}

\begin{proof}[Proof of Proposition \ref{prop:ConvDer}]
For all $j=1,\dots,d$ and $\alpha\in\N^d$ with $|\alpha|<m$, we have by
induction:
\begin{align*}
\partial_{x_j}\partial^\alpha\left(f\ast_{\mu}^p g\right)(x)
&=
\partial_{x_j}\left(\int_{\R^d} f(y)\, |\det(\mu(y))|^{1/p}\,
\alpha(\mu(y))\, D^{|\alpha|} g\big(\mu(y)(x-y)\big)\, \mathrm dy \right)
\\
&=
\int_{\R^d} f(y)\, |\det(\mu(y))|^{1/p}
\Big[\alpha(\mu(y)),\mu(y)_{\bullet,j}\Big] D^{|\alpha|+1}
g\big(\mu(y)(x-y)\big)\, \mathrm dy
\\
&=
\int_{\R^d} f(y)\, |\det(\mu(y))|^{1/p}\,
(\alpha+e_j)(\mu(y))\, D^{|\alpha|+1}
g\big(\mu(y)(x-y)\big)\, \mathrm dy
\\
&=
\left[\left(f\cdot (\alpha + e_j)(\mu)\right)\ast_{\mu}^p\,
D^{|\alpha|+1}g\right](x).
\end{align*}
\end{proof}

\begin{proof}[Proof of Proposition \ref{prop:convolutionContinuity2}]
Without loss of generality, we may assume that $\det\mu_t(x) >0$ for all $x\in\R^d$ and
$t\in\R$.
We will use the abbreviations $\delta_t:=\left(\det\mu_t\right)^{1/p}$, $f_\mu(x,y):=f\big(\mu(y)(x-y)\big)$ for
functions $\mu\colon \R^d\to\gldr$, $f\colon \R^d\to\R$ (note that latter notation
differs by a prefactor from the one used in Definition \ref{def:adaptiveConvolution}) and
$A_{k,\Cdot}^\intercal$ will denote the transpose of the $k$-th row of the matrix $A$.
The observations
\begin{small}
\begin{align*}
\left(\tr\left[\mu_t^{-1}(y)\partial_{k}\mu_t(y)\right]
\right)_{k=1}^d j_t(y)
&=
\sum_{k=1}^d \tr\left[\mu_t^{-1}(y)\partial_{k}\mu_t(y)\right]\, j_{k,t}(y)
=
\tr\left[\mu_t^{-1}(y)\sum_{k=1}^d\partial_{k}\mu_t(y)\, j_{k,t}(y)\right]
\\
&=
\tr\left[\mu_t^{-1}(y)
\begin{pmatrix}
\sum_{k=1}^d j_{k,t}(y)\,
\left(\partial_{k}\left(\mu_t\right)_{1,\Cdot}^\intercal\right)^{\intercal}(y)
\\
\vdots
\\
\sum_{k=1}^d j_{k,t}(y)\,
\left(\partial_{k}\left(\mu_t\right)_{d,\Cdot}^\intercal\right)^{\intercal}(y)
\end{pmatrix}
\right]
=
\tr\left[\mu_t^{-1}(y)
N_t(y)\right],
\\[0.2cm]
D_y\big[\mu_t(y)\, (x-y)\big]
&= 
\begin{pmatrix}
D_y\big[\left(\mu_{t}\right)_{1,\Cdot}(y)(x-y)\big]
\\ \vdots \\
D_y\big[\left(\mu_{t}\right)_{d,\Cdot}(y)(x-y)\big]
\end{pmatrix}
=
\underbrace{\begin{pmatrix}
(x-y)^{\intercal} \left(D_y
\left(\mu_{t}\right)_{1,\Cdot}^{\intercal}\right)(y)
\\ \vdots \\
(x-y)^{\intercal} \left(D_y
\left(\mu_{t}\right)_{d,\Cdot}^{\intercal}\right)(y)
\end{pmatrix}}_{=:\, M_t(x,y)}
- \mu_t(y),
\\[0.2cm]
M_t(x,y)\, j_t(y)
&=
\begin{pmatrix}
j_t(y)^{\intercal} \left(D_y
\left(\mu_{t}\right)_{1,\Cdot}^{\intercal}\right)^{\intercal}\hspace{-0.1cm}(y)\,
(x-y) \\ \vdots \\
j_t(y)^{\intercal} \left(D_y
\left(\mu_{t}\right)_{d,\Cdot}^{\intercal}\right)^{\intercal}\hspace{-0.1cm}(y)\,
(x-y)
\end{pmatrix}
=
N_t(y)\, (x-y)
\end{align*}
\end{small}
lead to
\begin{align*}
&\left(\nabla_y\Big[\delta_t(y)  \, g_{\mu_t}(x,y)\Big]\right)^{\intercal}j_t(y)
\\
&=
\delta_t(y) \left(
g_{\mu_t}(x,y) \left(\tr\left[\mu_t^{-1}(y)\partial_{k}\mu_t(y)\right]
\right)_{k=1}^d + (\nabla g)_{\mu_t}(x,y)^{\intercal}
\left[M_t(x,y)-\mu_t(y)\right] \right) j_t(y)
\\
&=
\delta_t(y) \Big(
g_{\mu_t}(x,y) \tr\left[\mu_t^{-1}(y)\, N_t(y)\right]
 + (\nabla g)_{\mu_t}(x,y)^{\intercal} \big[N_t(y)\, (x-y)
 - \mu_t(y)\, j_t(y)\big]
\Big)
\\
&=
\diver_x \Big(
\delta_t(y)\, g_{\mu_t}(x,y)\big[\mu_t^{-1}(y)\, N_t(y)\, (x-y)
 - j_t(y)\big]
 \Big)
\end{align*}
and
\begin{align*}
\partial_t\Big[\delta_t(y)  \, g_{\mu_t}(x,y)\Big]
&=
\delta_t(y) \left(\tr\left[\mu_t^{-1}(y)
\partial_t\mu_t(y)\right]\,
g_{\mu_t}(x,y) + (\nabla g)_{\mu_t}(x,y)^{\intercal}\partial_t\mu_t(y)\,
(x-y) \right)
\\
&=
\diver_x \Big(
\delta_t(y)\, g_{\mu_t}(x,y)\, \mu_t^{-1}(y)\, \partial_t\mu_t(y)\, (x-y)
 \Big)\, .
\end{align*}
Combining these two, we get:
\begin{align*}
\partial_t \rho_{g,t}(x)
&=
\int_{\R^d} \underbrace{\partial_t\rho_t(y)}_{= -\diver j_t(y)}
\delta_t(y) \, g_{\mu_t}(x,y)
+ \partial_t\Big[\delta_t(y)  \, g_{\mu_t}(x,y)\Big]\, \rho_t(y)
\, \mathrm dy
\\
&=
\int_{\R^d} \nabla_y\Big[\delta_t(y)  \, g_{\mu_t}(x,y)\Big]^{\intercal}j_t(y)
+ \partial_t\Big[\delta_t(y)  \, g_{\mu_t}(x,y)\Big]\, \rho_t(y)
\, \mathrm dy
\\
&=
-\diver \int_{\R^d} \delta_t(y)\, g_{\mu_t}(x,y)
\left(
j_t(y)
-
\mu_t^{-1}(y)
\Big[N_t(y) + \rho_t(y)\, \partial_t\mu_t(y)\Big](x-y)
\right) \mathrm dy
\\
&=
-\diver j_{g,t}(x)\, .
\end{align*}
The existence of all integrals follows directly from the assumptions and
Corollary \ref{cor:young}, while
$(j_{g,t})_{t\in[0,\infty)} \in
C^1\left([0,\infty)\times\R^d,\R^d\right)$ follows from Proposition $\ref{prop:ConvDer}$.
\end{proof}

\begin{proof}[Proof of Theorem \ref{theorem:AdaptationConditions}]
Again, let $g(x) = \gamma(\|x\|^2)$ for some function $\gamma\colon \R\to\R$. We have:
\begin{align*}
&(f(\Cdot\, - a)\ast^p_{\mu_{f(\Cdot\, - a)}} g) (x)
=
\int f(y-a) \left|\det\left(\mu_{f}(y-a)\right)\right|
g\left(\mu_{f}(y-a)(x-y)\right)\, \mathrm dy
\\
& \hspace{1cm} =
\int f(y) \left|\det\left(\mu_{f}(y)\right)\right|
g\left(\mu_{f}(y)(x-a-y)\right)\, \mathrm dy
=
(f\ast^p_{\mu_f}g) (x-a),
\\[0.2cm]
&(\alpha f\ast^p_{\mu_{\alpha f}} g) (x)
=
\int \alpha f(y) \left|\det\left(\mu_{f}(y)\right)\right|
g\left(\mu_{f}(y)(x-y)\right)\, \mathrm dy
=
\alpha (f\ast^p_{\mu_f}g) (x),
\\[0.2cm]
&(f(A\cdot\Cdot)\ast^p_{\mu_{f(A\cdot\Cdot)}} g) (x)
=
\int f(Ay) \left|\det\left(\mu_{f(A\cdot\Cdot)}(y)\right)\right|
g\left(\mu_{f(A\cdot\Cdot)}(y)(x-y)\right)\, \mathrm dy
\\
& \hspace{1cm} =
\int f(Ay) \left|\det\left(\mu_f(A
y)A\right)\right| \gamma\big((x-y)^{\intercal}A^{\intercal}\mu_f(A
y)^{\intercal}\mu_f(A y)A(x-y)\big)\, \mathrm dy
\\
& \hspace{1cm} =
\int f(y) \left|\det\left(\mu_f(y)\right)\right| \gamma\big((Ax-y)^{\intercal}
\mu_f(y)^{\intercal}\mu_f(y)(Ax-y)\big)\, \mathrm dy
\\
& \hspace{1cm} =
\int f(y) \left|\det\left(\mu_f(y)\right)\right| g\big( \mu_f(y)(Ax-y)\big)\,
\mathrm dy
=
(f\ast^p_{\mu_f}g) (Ax).
\end{align*}
\begin{align*}
&(f^{(t)}\ast_{\mu_{f^{(t)}}}^p g)(x+a_k^{(t)})
=
\int f^{(t)}(y) \, \abs{\det\mu_{f^{(t)}}(y)}^p
g\left(\mu_{f^{(t)}}(y)\big(x+a_k^{(t)}-y\big) \right)\mathrm dy
\\
&\hspace{1cm} = 
\sum_{j=1}^{K} \int f_j\big(y+a_k^{(t)}-a_j^{(t)}\big) \,
\abs{\det\mu_{f^{(t)}}\big(y+a_k^{(t)}\big)}^p
g\left(\mu_{f^{(t)}}\big(y+a_k^{(t)}\big)(x-y) \right)\mathrm dy
\\
& \hspace{0.4cm} \xrightarrow{t\to\infty}
\int f_k(y)\,
\abs{\det\mu_{f_k}(y)}^p
g\left(\mu_{f_k}(y)(x-y) \right)\mathrm dy
= 
(f_k\ast_{\mu_{f_k}}g)(x).
\end{align*}
\end{proof}

\begin{proof}[Proof of Proposition \ref{prop:muFourier}]
The proof is analogous to the one of Proposition
\ref{prop:muWigner2}.
\end{proof}

\begin{proof}[Proof of Proposition \ref{prop:muFBI}]
The proof is analogous to the one of Proposition
\ref{prop:muWigner2}.
\end{proof}

\begin{proof}[Proof of Proposition \ref{prop:AdaptiveMuFBI}]
The proof is analogous to the one of Proposition \ref{prop:muWigner2}.
Alternatively, it follows from Proposition \ref{prop:muFBI}.
\end{proof}

\begin{proof}[Proof of Theorem
\ref{theorem:AdaptationConditionsFulfilledAdaptiveMu}]
Since $(\phi(\Cdot - a)\ast\psi) (x) = (\phi\ast\psi) (x-a)$, the right hand
side of \eqref{equ:adaptiveFBImu} is translation-invariant for any choice of
$\mu_f$, which proves (A1).
If $\mu_f$ solves \eqref{equ:adaptiveFBImu}, then it also solves
\eqref{equ:adaptiveFBImu} for $\tilde f = \alpha f$ ($\alpha >0$) in place of
$f$, proving (A2).
For (A3) let $\mu_f$ be the solution of \eqref{equ:adaptiveFBImu}, $\tilde
f(x) = f(Ax)$ for some $A\in\gldr$ and $\mu_{\tilde f}(x):=
\sqrt{A^{\intercal}\, \mu_f(Ax)^{\intercal}\, \mu_f(Ax)\, A}$.
Then we have for $x,y\in\R^d$
\[
\big(\nabla \tilde f\nabla \tilde f^{\intercal} - \tilde f\, D^2 \tilde
f\big)(x)
=
A^\intercal\, \big(\nabla f\nabla f^{\intercal} - f\, D^2 f\big)(Ax)\, A,
\qquad
G_{\mu_{\tilde f}^{-2}(x)}(y)
=
\abs{\det A}\, G_{\mu_f^{-2}(Ax)}(Ay),
\]
and, since for any functions $\phi,\, \psi$, for which the convolution
$\phi\ast \psi$ exists, we have
\[
\left(\phi(A\cdot\Cdot)\ast \psi(A\cdot\Cdot)\right)(x)
=
\abs{\det A}^{-1}(\phi\ast \psi)(Ax),
\]
$\mu_{\tilde f}$ solves \eqref{equ:adaptiveFBImu} for $\tilde f$ in
place of $f$.
To prove (A4), let $R[f,\mu](x)$ denote the right-hand side of
\eqref{equ:adaptiveFBImu} with an arbitrary $\mu$ in place of $\mu_f$.
Adopting the notation of the Adaptation Axioms \ref{cond:adaptation}, we have
\[
R[f^{(t)}(\Cdot + a_k^{(t)}),\mu](x) \xrightarrow{t\to\infty} R[f_k,\mu](x)
\]
since each $f_k\in W^{2,2}(\R^d,\R)$. Hence, for $\tilde f^{(t)} = f^{(t)}(\Cdot
+ a_k^{(t)})$, the solution of the the implicit formula $\mu^2 = R[\tilde
f^{(t)},\mu]$ is asymptotically given by $\mu = \mu_{\tilde
f^{(t)}} = \mu_{f_k}$.
Therefore,
\[
\mu_{f^{(t)}}(x+a_k^{(t)})
\stackrel{(A1)}{=}
\mu_{\tilde f^{(t)}}(x)
\xrightarrow{t\to\infty}
\mu_{f_k}(x).
\]
\end{proof}

\begin{proof}[Proof of Proposition \ref{prop:muWigner2}]
As the Wigner transform is real-valued, we get for real-valued functions~$f$:
\begin{align*}
W f(x,-\xi)
&=
(2\pi)^{-d}\int_{\R^d}f\left(x+\frac{y}{2}\right)\,
f\left(x-\frac{y}{2}\right)\, e^{-iy^\intercal\xi}\, \mathrm dy
\\[0.1cm]
&=
 (2\pi)^{-d}\overline{\int_{\R^d}f\left(x+\frac{y}{2}\right)\,
f\left(x-\frac{y}{2}\right)\, e^{iy^\intercal\xi}\, \mathrm dy}
=
\overline{W f(x,\xi)}
=
W f(x,\xi),
\end{align*}
and therefore the expectation value of $\P_{\rho_x}$ vanishes.
For the covariance matrix, we use the transformation
\[
z_1 = y_1-y_2,\qquad z_2 = y_1+y_2
\]
and the function
\[
F(z_1,z_2) = f\left(x+\frac{z_2 + z_1}{4}\right)f\left(x-\frac{z_2 +
z_1}{4}\right)f\left(x+\frac{z_2 - z_1}{4}\right)f\left(x-\frac{z_2 -
z_1}{4}\right)
\]
to compute:
\begin{small}
\begin{align*}
\int_{\R^d }|W & f|^2(x,\xi)\, \mathrm d\xi
=
(2\pi)^{-2d} \int_{\R^{3d} } f\left(x+\tfrac{y_1}{2}\right)
f\left(x-\tfrac{y_1}{2}\right)f\left(x+\tfrac{y_2}{2}\right)
f\left(x-\tfrac{y_2}{2}\right) e^{i(y_1-y_2)^\intercal\xi}\,
\mathrm dy_1 \mathrm dy_2 \mathrm d\xi
\\
&=
\frac{(2\pi)^{-2d}}{2^d} \int_{\R^{3d}} F(z_1,z_2)\, 
e^{iz_1^\intercal\xi}\, \mathrm dz_1 \, \mathrm d\xi\, \mathrm dz_2
=
\frac{(2\pi)^{-d}}{2^d} \int F(0,z_2)\, \mathrm dz_2
\\
&=
\frac{(2\pi)^{-d}}{2^d} \int f^2\left(x+\tfrac{z_2}{4}\right)\,
f^2\left(x-\tfrac{z_2}{4}\right)\, \mathrm dz_2
=
2^d\, (2\pi)^{-d} \int f\left(z\right)^2\,
f\left(2x-z\right)^2\, \mathrm dz
\\
&=
2^d\, (2\pi)^{-d}\, \left(f^2\right)\ast\left(f^2\right) (2x),
\\[0.2cm]
\int_{\R^d } |W & f|^2 (x,\xi)\, \xi\xi^{\intercal}\, \mathrm d\xi
=
(2\pi)^{-2d} \int_{\R^{3d} } f\left(x+\tfrac{y_1}{2}\right)
f\left(x-\tfrac{y_1}{2}\right)f\left(x+\tfrac{y_2}{2}\right)
f\left(x-\tfrac{y_2}{2}\right) e^{i(y_1-y_2)^\intercal\xi}\, \xi\xi^{\intercal}\, 
\mathrm dy_1 \mathrm dy_2 \mathrm d\xi
\\
&=
\frac{(2\pi)^{-2d}}{2^d} \int_{\R^{3d}} F(z_1,z_2)\, 
e^{iz_1^\intercal\xi}\, \xi\xi^{\intercal}\, \mathrm dz_1 \, \mathrm d\xi\,
\mathrm dz_2
=
-\frac{(2\pi)^{-d}}{2^d} \int D_{z_1}^2 F(0,z_2)\, \mathrm dz_2
\\
&=
\frac{(2\pi)^{-d}}{2^{d+3}} \int
f^2\left(x+\tfrac{z_2}{4}\right)
\left[
\nabla f\nabla f^{\intercal} - fD^2f
\right]\left(x-\tfrac{z_2}{4}\right)
+
f^2\left(x-\tfrac{z_2}{4}\right)
\left[
\nabla f\nabla f^{\intercal} - fD^2f
\right]\left(x+\tfrac{z_2}{4}\right)
\mathrm dz_2
\\
&=
\frac{(2\pi)^{-d}}{2^{d+3}} \int
f^2(z)
\left[
\nabla f\nabla f^{\intercal} - fD^2f
\right]\left(2x-z\right)
+
f^2\left(2x-z\right)
\left[
\nabla f\nabla f^{\intercal} - fD^2f
\right]\left(z\right)
\mathrm dz_2
\\
&=
2^{d-2}\, (2\pi)^{-d}\, \left(f^2\right)\ast\left[
\nabla f\nabla f^{\intercal} - fD^2f
\right] (2x).
\end{align*}
\end{small}
Taking the quotient proves the formula for the covariance matrix.
\end{proof}

\begin{proof}[Proof of Theorem \ref{theorem:AdaptationConditionsFulfilled}]
Adaptation Axiom \ref{cond:adaptation} (A1) follows from
\begin{align*}
\left(f(\Cdot -a)\ast g(\Cdot -b)\right)(x)
&=
\int f(y-a)\, g(x-y-b)\, \mathrm dy
=
\int f(y)\, g(x-(a+b)-y)\, \mathrm dy
\\
&=
(f\ast g)(x-(a+b)).
\intertext{
(A2) is straightforward and (A3) follows from
}
\left(f(A\cdot\Cdot)\ast g(A\cdot\Cdot)\right)(x)
&=
\int f(Ay)\, g(A(x-y))\, \mathrm dy
=
|\det A|^{-1}\int f(y)\, g(Ax-y)\, \mathrm dy
=
\frac{(f\ast g)(Ax)}{|\det A|}
\intertext{
in the following way:
}
\mu_{f(A\cdot\Cdot)}^{(d)}(x)^{\intercal}\, \mu_{f(A\cdot\Cdot)}^{(d)}(x)
&=
\frac{f^2(A\cdot\Cdot)\ast\left[A^{\intercal}\, \nabla f(A\cdot\Cdot)\, \nabla
f^{\intercal}(A\cdot\Cdot)A - A^{\intercal}\, f\,
D^2f(A\cdot\Cdot)\, A\right]}{f^2(A\cdot\Cdot)\ast f^2(A\cdot\Cdot)}\, (2x)
\\
&=
\frac{A^{\intercal}\left[f^2\ast\left(\nabla f\, \nabla f^{\intercal} -
f\, D^2f\right)\right]A}{f^2\ast f^2}\, (2Ax)
=
A^{\intercal}\, \mu_f^{(d)}(Ax)^{\intercal}\, \mu_f^{(d)}(Ax)\, A\, .
\end{align*}
\end{proof}

\begin{proof}[Proof of Corollary \ref{cor:gaugeWidth2}]
A simple computation shows:
\begin{align*}
Df(x)
&=
-f(x)(x-a)^{\intercal}\Sigma^{-1}
\\
D^2f(x)
&=
f(x)\left[\Sigma^{-1}(x-a)(x-a)^{\intercal}\Sigma^{-1} - \Sigma^{-1}\right]
\\
\left(\nabla f\, \nabla f^{\intercal} - f\, D^2f\right)(x)
&=
f^2(x)\, \Sigma^{-1}.
\end{align*}
The claim follows from the definitions \eqref{equ:adaptiveFBImu} of
$\mu_f^{(d)}$ and \eqref{equ:finalChoiceMu} of
$\mu_f^{(e)}$.
\end{proof}


\begin{proof}[Proof of Proposition \ref{prop:young2}]
For the $h$-adaptive convolution of type two the property $\left\|
G(\Cdot,y)\right\|_p = \left\|g\right\|_p$ is straightforward for all
$y\in\R^d$, $1\le p \le\infty$ and Theorem \ref{theorem:young1} proves the
claim.
For the $h$-adaptive convolution of type three, we denote
\[
\gamma(x,z) := \frac{g_2(z-h(x))}{\left\|g_2(z-h(\Cdot))\right\|_p}
\]
and observe for $y\in\R^d$ and $p = 1,\, \infty$:
\begin{small}
\begin{align*}
\left\|\tilde G(\Cdot,y)\right\|_1
&\le
\left\|g_2\right\|_1 \int \int |g_1(z-h(y))\, \gamma(x,z)|\, \mathrm
dz\, \mathrm dx
=
\left\|g_2\right\|_1 \int |g_1(z-h(y))| \underbrace{\int
|\gamma(x,z)|\, \mathrm dx}_{=1} \mathrm dz
=
\left\|g_1\right\|_1 \left\|g_2\right\|_1,
\\
\left\|\tilde G(\Cdot,y)\right\|_\infty
&=
\left\|g_2\right\|_\infty \, \operatorname*{ess\,sup}_{x\in\R^d} \left| \int
g_1(z-h(y))\, \gamma(x,z) \, \mathrm dz \, \right|
\le
\left\|g_2\right\|_\infty \int |g_1(z-h(y))|\, \mathrm dz
=
\left\|g_1\right\|_1 \left\|g_2\right\|_\infty.
\end{align*}
\end{small}
For $1<p<\infty$, let $p'$ denote the conjugate of $p$ (i.e. $1/p + 1/p' =
1$). H\"older's inequality yields
\begin{align*}
|\tilde G(x,y)|
&\le
\left\|g_2\right\|_p \int |g_1(z-h(y))|^{1/p'}\, |g_1(z-h(y))|^{1/p}
\gamma(x,z)\, \mathrm dz
\\
&\le
\left\|g_2\right\|_p \left\|g_1(\Cdot-h(y))^{1/p'}\right\|_{p'}
\left\|g_1(\Cdot-h(y))^{1/p}\, \gamma(x,\Cdot)\right\|_{p}\, ,
\end{align*}
which implies for each $y\in\R^d$,
\begin{small}
\begin{align*}
\left\|\tilde G(\Cdot,y)\right\|_p^p
&\le
\left\|g_2\right\|_p^p \left\|g_1^{1/p'}\right\|_{p'}^p
\int \left\|g_1(\Cdot-h(y))^{1/p}\, \gamma(x,\Cdot)\right\|_{p}^p\, \mathrm
dx
\le
\left\|g_2\right\|_p^p \left\|g_1\right\|_1^{p/p'}
\int\int |g_1(z-h(y))\, \gamma(x,z)^p|\, \mathrm dz\, \mathrm dx
\\
&=
\left\|g_2\right\|_p^p \left\|g_1\right\|_1^{p/p'}
\int |g_1(z-h(y))| \underbrace{\int |\gamma(x,z)|^p\, \mathrm dx}_{=1} \mathrm
dz
=
\left\|g_2\right\|_p^p\left\|g_1\right\|_1^{p/p'}\left\|g_1\right\|_1
=
\left\|g_1\right\|_1^{p}\left\|g_2\right\|_p^p.
\end{align*}
\end{small}
Therefore $\left\| \tilde G(\Cdot,y)\right\|_p \le \left\|g_1\right\|_1
\left\|g_2\right\|_p$ (for all $y\in\R^d$ and $1\le p\le \infty$) also holds for
type three and again Theorem \ref{theorem:young1} proves the claim.
\end{proof}

\begin{proof}[Proof of Proposition \ref{prop:symmetricG}]
Since in this case $g_1=g_2=g$, the symmetry of $\tilde G$ follows from
\[
\tilde G(x,y)
=
\left\| g \right\|_p \int g(z-h(y))\,
\frac{g(z-h(x))}{\left\|g(z-h(\Cdot))\right\|_p} \, \mathrm dz
=
\left\| g \right\|_p \int \frac{g(z-h(y))}{\left\|g(z-h(\Cdot))\right\|_p}\,
g(z-h(x)) \, \mathrm dz
=
\tilde G(y,x).
\]
Following the proof of Proposition \ref{prop:young2}, we conclude that for each
$1\le p\le\infty$ both $\left\| \tilde G(\Cdot,y)\right\|_p \le
\left\|g\right\|_1\left\|g\right\|_p$ for each $y\in\R^d$ and $\left\| \tilde
G(x,\Cdot)\right\|_p \le \left\|g\right\|_1\left\|g\right\|_p$ for each
$x\in\R^d$. Theorem \ref{theorem:young2} proves the claim.
\end{proof}